\documentclass[10pt]{article}
\usepackage{geometry}

\usepackage{amssymb}
\usepackage{amsmath}
\usepackage{amsthm}
\usepackage[dvipsnames]{xcolor}
\usepackage[shortlabels]{enumitem}
\usepackage[normalem]{ulem}
\usepackage{thmtools, thm-restate}
\usepackage{hyperref}
\usepackage{cleveref}
\usepackage{comment}

\newcommand{\CC}{\mathbb{C}}
\newcommand{\RR}{\mathbb{R}}
\newcommand{\ZZ}{\mathbb{Z}}
\newcommand{\NN}{\mathbb{N}}
\newcommand{\Symm}{\mathfrak{S}}
\newcommand{\op}{\operatorname}
\newcommand{\fix}{\op{fix}}
\newcommand{\GL}{\op{GL}}
\newcommand{\codim}{\op{codim}}
\newcommand{\codimfix}{\codim\fix}
\newcommand{\cdf}{\mathrm{cdf}}
\newcommand{\id}{\mathrm{id}}
\newcommand{\lr}{\ell_R}
\newcommand{\lR}{\lr}
\newcommand{\supp}{\operatorname{Supp}}
\newcommand{\rank}{\operatorname{rank}}
\newcommand{\wt}{\operatorname{wt}}

\newcommand{\vmax}{v_{\mathrm{max}}}
\newcommand{\ol}{\overline}
\newcommand{\exc}{\op{exc}}
\newcommand{\Ra}{{\,$\Rightarrow$\,}}

\newtheorem{thm}{Theorem}[section]
\newtheorem{prop}[thm]{Proposition}
\newtheorem{lem}[thm]{Lemma}
\newtheorem{lemma}[thm]{Lemma}
\newtheorem{cor}[thm]{Corollary}

\theoremstyle{definition}
\newtheorem*{defn}{Definition}

\newtheorem{ex}[thm]{Example}
\newtheorem{remark}[thm]{Remark}

\newcommand{\FindStat}[1]{\href{https://www.findstat.org/StatisticsDatabase/#1}{#1}}

\title{Coincidences between intervals in two partial orders on complex reflection groups}
\author{Joel Brewster Lewis and Jiayuan Wang}

\begin{document}
\maketitle

\begin{abstract}
    In a finite real reflection group, the reflection length of each element is equal to the codimension of its fixed space, and the two coincident functions determine a partial order structure called the absolute order.  In complex reflection groups, the reflection length is no longer always equal to the codimension of fixed space, and the two functions give rise to two different partial orders on the group.  We characterize the elements $w$ in the combinatorial family $G(m, p, n)$ of complex reflection groups for which the intervals below $w$ in these two posets coincide. We also explore the relationship between this property and other natural properties of elements in complex reflection groups; some general theory of posets arising from subadditive functions on groups; and the particular case of subadditive functions on the symmetric group.
\end{abstract}

\section{Introduction}

Suppose that $G$ is a group and $f: G \to \RR_{\geq 0}$ is a subadditive function (that is, $f(xy) \leq f(x) + f(y)$ for all $x, y \in G$) such that $f(x) = 0$ if and only if $x$ is the identity in $G$.  It is easy to show (see Proposition~\ref{prop:poset}) that such a function gives rise to a partial order $\leq_f$ on $G$: one has
\[
x \leq_f y \qquad \Longleftrightarrow \qquad f(x) + f(x^{-1}y) = f(y).
\]
Two sources of such functions $f$ naturally present themselves, one algebraic and one geometric.  First, if $T$ is any generating set of $G$, then the \emph{$T$-length}
\begin{align*}
\ell_T: G & \to \NN \\
x & \mapsto \min\{k : \exists t_1, \ldots, t_k \in T \text{ such that } x = t_1 \cdots t_k\}
\end{align*}
has the requisite properties.  The resulting partial order $\leq_{\ell_T}$ may equivalently be characterized by saying that $x \leq_{\ell_T} y$ if $x$ lies along a path of minimum length from the identity to $y$ in the Cayley graph of $G$ generated by $T$, or that each minimum-length $T$-word $x = t_1 \cdots t_{\ell_T(x)}$ for $x$ can be extended to a minimum-length $T$-word $y = t_1 \cdots t_{\ell_T(y)}$ for $y$.

Second, if $G$ acts on a finite-dimensional vector space $V$ (i.e., by choice of a linear representation), with the element $g$ having \emph{fixed space} $\fix(g) := \ker(g - 1)$, one has \cite[Prop.~2.9]{HLR} that the fixed space codimension $\codimfix(g) := \dim V - \dim \fix(g)$ is subadditive.  Moreover, we have $\codimfix(g) = 0 \Longleftrightarrow x = \id$ precisely when the representation of $G$ is faithful.  In this case, we denote by $\leq_{\cdf}$ the resulting partial order.

It is reasonable to inquire when the two subadditive functions, and hence the two partial orders, just defined (one by a set of generators, the other by a choice of representation) coincide.  In \cite[Prop.~2.11]{HLR}, it was shown that a necessary condition for this equality is that $G$ be a reflection group, with $T = R$ its subset of reflections.  In this case, it's easy to see that 
\begin{equation}
\label{eq:codim < lR}
    \codimfix(g) \leq \lR(g) \text{ for all } g \in G.  
\end{equation}
It has been known for fifty years that if $G$ is a \emph{real} reflection group (i.e., a finite Coxeter group), then in fact $\codimfix(g) = \lR(g)$ for all $g$ in $G$ \cite[Lem.~2]{Carter}.  The same holds true in the complex reflection group $G(m, 1, n)$ (the wreath product $(\ZZ/m\ZZ) \wr \Symm_n$) \cite[Rem.~2.3(1)]{Shi2007}, as well as various other natural groups \cite{BradyWatt,Dieudonne, HLR}, but it does \emph{not} hold in the other finite complex reflection groups \cite{FosterGreenwood}.

Even when the whole partial orders $(G, \leq_{\cdf})$ and $(G, \leq_{\lR})$ do not coincide, certain important pieces of them may.  In \cite[Cor.~6.6]{LM2021}, it was shown that when $G$ is a well generated complex reflection group and $c$ is a Coxeter element in $G$, the two intervals $[\id, c]_{\cdf}$ and $[\id, c]_{\lR}$ are identical---they are the \emph{noncrossing partition lattice} of $G$.  This naturally raises the question \cite[Q.~8.11]{LM2021} of which other elements have this property.  The main result of this paper is to characterize these elements in the infinite family $G(m, p, n)$ of irreducible complex reflection groups.  We do this in terms of the combinatorial description of the groups $G(m, p, n)$---see Section~\ref{sec:background} for details.

\begin{restatable}{thm}{maintheorem}
\label{thm:main}
An element $w \in G(m, p, n)$ satisfies $[\id, w]_{\lr} = [\id, w]_{\cdf}$ if and
only if (1) the cycle weights of $w$ that are not $0 \pmod{p}$ can be partitioned into pairs that sum to $0$ (equivalently, $\lR(w) = \codimfix(w)$),
\textup{and} (2) any subset of cycle weights that sums to $0 \pmod{p}$ is a disjoint union of some weights that are $0 \pmod{p}$ and some pairs of weights that sum to $0$.
\end{restatable}

The bulk of this paper is devoted to the proof of Theorem~\ref{thm:main}.  In Section~\ref{sec:background}, we introduce the necessary background definitions and notations, and discuss some general properties of posets determined by subadditive functions on groups.  After this, we divide the proof of Theorem~\ref{thm:main} in several stages.  In Section~\ref{sec:necessary}, we prove the given conditions are necessary by explicitly constructing elements that belong to one interval but not the other when the conditions are not met.  To show that they are sufficient, we develop in Section~\ref{sec:sufficient} a detailed combinatorial description of the interval $[\id, w]_{\cdf}$ for $w \in G(m, p, n)$, allowing us to establish that $u \in [\id, w]_{\cdf} \Longrightarrow u \in [\id, w]_{\lR}$ when $w$ satisfies the necessary conditions.  We end in Section~\ref{sec:further remarks} with a number of closing remarks, including explorations of the heritability of the relation $[\id, w]_{\lR} = [\id, w]_{\cdf}$ in arbitrary complex reflection groups, the extent to which special classes of elements (the regular elements and the parabolic quasi-Coxeter elements) have this property, and the question of which other permutation statistics on the symmetric group $S_n$ are subadditive.

\section{Background}
\label{sec:background}

\subsection{The infinite family of complex reflection groups}

Say that a linear transformation $t$ on a vector space $V$ is a \emph{reflection} if its \emph{fixed space} 
\[
\fix(t) := \{v \in V: t(v) = v\} = \ker(t - 1)
\]
has codimension $1$, and that a finite subgroup $W \subset \GL(V)$ is a \emph{reflection group} if $W$ is generated by its subset of reflections.  Complex reflection groups (i.e., those for which the field of scalars of $V$ is $\CC$) were classified by Shephard and Todd \cite{ShephardTodd}: every complex reflection group is a direct product of irreducible groups, and each irreducible group either belongs to an infinite family $G(m, p, n)$ for positive integers $m$, $p$, $n$ with $p \mid m$, or is one of $34$ exceptional examples.\footnote{This description elides a few technicalities; among these is that the definition of $G(1, 1, n)$ given below (as the set of $n \times n$ permutation matrices) is not irreducible in its action on $\CC^n$, because it acts trivially on the subspace $\op{span}\{ (1, 1, \ldots, 1)\}$.  These technicalities will not play a role in what follows.}  This paper is concerned primarily with the groups of the infinite family; we describe them now.

By appropriate choice of basis, the group $G(m, 1, n)$ may be realized concretely as the group of $n \times n$ monomial matrices whose nonzero entries are $m$th roots of $1$.  Algebraically, $G(m, 1, n)$ is the wreath product $(\ZZ/m\ZZ) \wr \Symm_n$ of the cyclic group of order $m$ with the symmetric group $\Symm_n$.  Thus, its elements may be represented by a pair $w = [u; a]$ where $u \in \Symm_n$ is the \emph{underlying permutation} of $w$ and $a = (a_1, \ldots, a_n) \in (\ZZ/m\ZZ)^n$ is its tuple of \emph{weights}.  A \emph{cycle} of $w$ simply means a cycle of its underlying permutation; in particular, we consider fixed points to be cycles (of size $1$), and we denote by $c(w)$ the number of cycles of $w$.
For any subset $I \subseteq [n]$, we say that the weight of $I$ (relative to $w$) is $\sum_{i \in I} a_i$.  This notion will be especially relevant when $I$ is (the underlying set of) a cycle of $w$ or a collection of cycles of $w$.  In particular, we denote by $\wt(w)$ the weight $\sum_{i = 1}^n a_i$ of $w$.  For $p \mid m$, the group $G(m, p, n)$ is the normal subgroup of $G(m, 1, n)$ consisting of all those elements whose weight is a multiple of $p$.

Let $\zeta = \exp(2\pi i/m)$.  If $w = [u; a]$ is an element of $G(m, p, n)$ and $(x_1 \cdots x_k)$ is a weight-$0$ cycle of $w$, the vector in $\CC^n$ whose $x_i$th entry is $\zeta^{a_{x_1} + \ldots + a_{x_{i-1}}}$ for $i = 1, \ldots, k$ and whose other entries are $0$ is easily seen to be fixed by the action of $w$.  In fact, the collection of such vectors (taken over all weight-$0$ cycles of $w$) span $\fix(w)$, and consequently $\codimfix(w) = n - c_0(w)$, where $c_0(w)$ represents the number of weight-$0$ cycles of $w$.  Considering the case that $\codimfix(w) = 1$, we see that there are two flavors of reflection in $G(m, p, n)$: first, for any $i \neq j$ in $\{1, \ldots, n\}$ and any $a \in \ZZ/m\ZZ$, the element
\[
[(i \ j); (0, \ldots, 0, a, 0, \ldots, 0, -a, 0, \ldots, 0)]
\]
with $a_i = a$ and $a_j = -a$ and $a_k = 0$ if $k \neq i, j$ is a \emph{transposition-like reflection}.  Second, if $p < m$, then for any $k$ in $\{0, 1, \ldots, m/p - 1\}$, the element
\[
[\id; (0, \ldots, 0, kp, 0, \ldots, 0)]
\]
is a \emph{diagonal reflection} (where the nonzero weight may occur in any of the $n$ positions).

Any real reflection group may be complexified by extension of scalars, yielding a complex reflection group.  In particular, the four infinite families of real reflection groups are all realized inside the infinite family $G(m, p, n)$:
\begin{itemize}
    \item the symmetric group $\Symm_n$ is $G(1, 1, n)$ (type A);
    \item the hyperoctahedral group of signed permutations of degree $n$ is $G(2, 1, n)$ (type B/C);
    \item its normal subgroup of even-signed permutations is $G(2, 2, n)$ (type D); and
    \item the dihedral group of order $2 \times m$ is $G(m, m, 2)$ (type I).
\end{itemize}

\subsection{Cycle partitions and reflection length}

In any reflection group $W$ with reflections $R$, the \emph{reflection length} $\lR(w)$ of an element $w$ is defined to be
\[
\lR(w) = \min\{k : \exists t_1, \ldots, t_k \in R \text{ s.t. } w = t_1 \cdots t_k\}.
\]
As mentioned in the introduction, when $W$ is a real reflection group or the group $G(m, 1, n)$, we have $\lR(w) = \codimfix(w)$ for all $w \in W$.  For the other groups $G(m, p, n)$ in the combinatorial family, a formula for reflection length was given by Shi.  In order to state it, we need some additional terminology.

Given a finite set $S$, a \emph{(set) partition} of $S$ is a collection of disjoint nonempty sets whose union is $S$.  The elements of the partition are called its \emph{parts}.  We use the following notation for set partitions:
$
[ 1 \ 3 \mid 2 \mid 4]
$
represents the set partition whose three parts are $\{1, 3\}$, $\{2\}$, and $\{4\}$.  The same set partition could be written many different ways, e.g., as $[2 \mid 3 \ 1 \mid 4]$.  

We will frequently deal with set partitions $\Pi$ of the set of cycles of an element $w$ of $G(m, p, n)$.
Let $w\in G(m,p,n)$ with cycles $C_1, \ldots, C_k$. We say that a set partition $\Pi$ on $C_1, \ldots, C_k$ is a \emph{null cycle partition}\footnote{In \cite{LW}, these partitions of the cycles were called simply ``cycle partitions''.} if, for every part in $\Pi$, the weights of its cycles sum to $0 \pmod{p}$. For every null cycle partition $\Pi$ of $w$, the \emph{value} $v(\Pi)$ is defined to be 
\[
v(\Pi):= |\Pi| + v_m(\Pi), 
\] 
where $|\Pi|$ denotes the number of parts of $\Pi$ and $v_m(\Pi)$ denotes the number of parts of $\Pi$ whose cycle weights sum to $0$ (not just $0\pmod{p})$. For a fixed element $w\in G(m,p,n)$, there can be many null cycle partitions with many different values. Let 
\[
\vmax(w):=\mathrm{max}\{v(\Pi): \Pi \text{ is a null cycle partition for } w\}
\] 
be the maximum value of any null cycle partition of $w$; we say that the cycle partitions that realize $\vmax(w)$ are its \emph{maximum (null) cycle partitions}.  Then we have the following formula for reflection length. 

\begin{thm}[Shi {\cite[ Thm.~4.4]{Shi2007}}]
\label{thm: reflection length}
   For $w\in G(m,p,n)$, its reflection length is \[\lr(w)=n+c(w)-\vmax(w).\] 
\end{thm}

By combining Theorem~\ref{thm: reflection length} with the formula for fixed space codimension, Shi was able to characterize the elements $w$ in $G(m, p, n)$ that satisfy $\codimfix(w) = \lR(w)$.

\begin{prop}[{\cite[ Prop.~5.3~(2)]{Shi2007}}]\label{prop: shi}
Let $w\in G(m,p,n)$. Then $\lr(w)=\codimfix(w)$ if and only if $w$ has a null cycle partition $\Pi$ in which the size $|B|$ of each part $B$ is at most $2$, and which further satisfies the following conditions: if $|B|=1$ then the cycle in $B$ has weight $0\pmod{p}$; if $|B|=2$, the cycles in $B$ have nonzero weights that sum to $0$. Moreover, in this case, the given cycle partition is maximum.
\end{prop}

The following rephrasing is perhaps more congenial to work with.

\begin{cor}\label{cor:shi}
An element $w \in G(m, p, n)$ satisfies $\lr(w) = \codimfix(w)$ if and
only if the multiset of cycle weights of $w$ that are not $0 \pmod{p}$ can be partitioned into pairs that sum to $0$. 
\end{cor}

\subsection{Posets from subadditive functions}
\label{sec:general}

We end this section by proving a few general theorems about the posets $(G, \leq_f)$ for a subadditive function $f$ on a group $G$.  The first is the result (mentioned in the introduction) that the subadditive functions we consider really do determine a poset structure.  As observed in \cite[Fn.~1]{FosterGreenwood}, the proof in general is essentially the same as the proof given in \cite[Prop.~3]{BradyWatt} that $\leq_\cdf$ is a partial order on the orthogonal group.

\begin{prop}\label{prop:poset}
    Let $G$ be any group and suppose that $f: G \to \RR_{\geq 0}$ is a subadditive function (i.e., $f(xy) \leq f(x) + f(y)$ for all $x, y \in G$) such that $f(x) = 0$ if and only if $x$ is the identity in $G$.  Define a relation $\leq_f$ on $G$ by 
    \[
    x \leq_f y \qquad \Longleftrightarrow \qquad f(x) + f(x^{-1}y) = f(y).
    \]
    Then $\leq_f$ is a partial order on $G$.
\end{prop}

\begin{proof}
    We have three properties to check.  
    
    Since $f(\id) = 0$, for any $x \in G$ we have $f(x) = f(x) + f(\id) = f(x) + f(x^{-1} x)$, and so $x \leq_f x$.

    Suppose $x \leq_f y$ and $y \leq_f x$.  Since $f$ takes only nonnegative values, we have
    \[
    f(x) \leq f(x) + f(x^{-1}y) = f(y) \leq f(y) + f(y^{-1}x) = f(x).
    \]
    This forces $f(x^{-1}y) = 0$; by the hypothesis on $f$, we have $x^{-1}y = \id$, so $x = y$.

    Finally, suppose $x \leq_f y$ and $y \leq_f z$.  By subadditivity of $f$ and the definition of $\leq_f$, we have 
    \begin{align*}
        f(z) & \leq f(x) + f(x^{-1}z) \\
            & \leq f(x) + f(x^{-1}y) + f(y^{-1}z) \\
            & = f(y)+ f(y^{-1}z) \\
            & = f(z).
    \end{align*}
    This forces $f(z) = f(x) + f(x^{-1}z)$, so $x \leq_f z$.
\end{proof}

For a given reflection group, the functions $\lR(-)$ and $\codimfix(-)$ are both constant on conjugacy classes.  For functions of this form, the poset $(G,\leq_f)$ carries additional symmetries.

\begin{prop}[{essentially \cite[Prop.~2.5]{HLR}}] \label{prop:conjugacy}
    Let $G$ be any group and $f: G \to \RR_{\geq 0}$ a subadditive function such that $f(x) = 0$ if and only if $x$ is the identity in $G$.  Suppose furthermore that $f$ is constant on conjugacy classes, i.e., $f(h^{-1}gh) = f(g)$ for all $g, h \in G$.  Then for any $x \leq_f z$ in $G$, the map $y \mapsto xy^{-1}z$ is a poset antiautomorphism of $[x, z]_f$ (so in particular $x y^{-1}z \in [x, z]_f$). 
\end{prop}

The proof is identical to the proof of \cite[Prop.~2.5]{HLR}, which considered the case $f = \ell_T$ for some generating set $T$ of $G$ but which used no hypotheses beyond the ones stated here.

When $f = \ell_T$ is the length function for a generated group $G$, the poset $(G, \leq_T)$ is automatically graded by $\ell_T$, i.e., there is a minimum element and each cover relation is between a pair of elements of $T$-lengths $k$ and $k + 1$ for some $k$.  In particular, for a complex reflection group $W$, when $\lR(w) = \codimfix(w)$ for all $w \in W$, we have that the atoms of the associated poset are precisely the reflections.  It was observed in \cite[Prop.~2.4]{FosterGreenwood} that the converse holds: whenever a complex reflection group $W$ satisfies $\lR(w) \neq \codimfix(w)$ for some $w \in W$, there exists a non-reflection atom in the cdf-poset.  We generalize this to give a following characterization of length functions for generated groups among all subadditive functions.

\begin{prop}\label{prop:graded}
    Let $G$ be a group and $f: G \to \NN$ a subadditive function such that $f(x) = 0$ if and only if $x$ is the identity in $G$.  The poset $(G, \leq_f)$ is $f$-graded, i.e.,
    \[
    x \lessdot y \quad \Longrightarrow \quad f(y) = f(x) + 1,
    \]
    if and only if $f = \ell_T$ is the length function for $G$ 
    with generating set $T := \{ t \in G : f(t) = 1\}$.
\end{prop}
\begin{proof}
Suppose that $f = \ell_T$ for some generating set $T$; we must show that $(G, \leq_f)$ is $f$-graded.  We immediately have that $f : G \to \NN$, $f(\hat{0})=0$, and $f(t)=1$ for $t \in T$.  Choose $x <_f y$, so that (by definition) $\ell_T(x) + \ell_T(x^{-1}y) = \ell_T(y)$.  If $\ell_T(x^{-1}y) > 1$ then we can write $x^{-1}y = z \cdot t$ for some $t \in T$ and some $z \in G$ with $\ell_T(z) = \ell_T(x^{-1}y) - 1 > 0$.  In this case, $x <_f xz <_f xzt = y$, so $y$ does not cover $x$ in the poset $(G, \leq_f)$.  Taking the contrapositive, if $x \lessdot y$ then $f(y)=f(x)+1$, as claimed.

Conversely, suppose that $(G, \leq_f)$ is $f$-graded, and let $T = \{ g \in G : f(g) = 1\}$.  We proceed by induction on $f(g)$.  By assumption $f(\id_G) = 0 = \ell_T(g)$.  Choose any non-identity element $g \in G$.  Since $f$ is $\NN$-valued, $\{ f(x) : x <_f g\}$ is a set of integers less than $f(g)$, so it has some maximum element $k$; choose $x <_f g$ with $f(x) = k$.  By the maximality of $k$, we must in fact have $x \lessdot g$, since any $y$ with $x <_f y <_f g$ would satisfy $f(x) < f(y) < f(g)$.  Thus $k = f(g) - 1$.  By the definition of $\leq_f$, we have
\[
f(g) = f(x) + f(x^{-1}g),
\]
so $f(x^{-1}g) = 1$, and $x^{-1}g \in T$ by definition.  By the inductive hypothesis, $x = t_1 \cdots t_k$ with $t_i \in T$ for $i = 1, \ldots, k$, and therefore $g = x \cdot x^{-1}g = t_1 \cdots t_k \cdot x^{-1}g$ is a product of $f(g) = k + 1$ elements of $T$.  This implies $f(g) \geq \ell_T(g)$.  On the other hand, by subadditivity $f(g) \leq f(t_1) + \ldots + f(t_k) + f(x^{-1}g) = k + 1$, so in fact $f(g) = \ell_T(g)$, as claimed.
\end{proof}

\begin{remark}\label{rem:metric}
    Subadditive functions of the kind in Propositions~\ref{prop:poset}, \ref{prop:conjugacy}, and~\ref{prop:graded} are closely related to group metrics, as in \cite{ChatterjeeDiaconis}.  Indeed, if $d(-, -)$ is a left-invariant metric on a group $G$ (so that $d(xy, xz) = d(y, z)$ for all $x, y, z \in G$), then the function $f(x) := d(\id, x)$ is a subadditive function on $G$ such that $f(x) = 0 \Leftrightarrow x = \id$; and conversely if $f$ is a subadditive function such that $f(x) = 0 \Leftrightarrow x = \id$ and with the additional symmetry $f(x) = f(x^{-1})$ for all $x \in G$, then $d(x, y) := f(x^{-1}y)$ is a left-invariant metric on $G$.
\end{remark}

We now move to the proof of the main theorem.

\section{Necessity}\label{sec:necessary}

In this section, we prove that the given conditions are necessary, i.e., that if the cycle weights of $w$ cannot be partitioned into pairs that sum to $0$ and singletons that are $0 \pmod{p}$, or if they can be so partitioned but there exists a subset of the cycle weights of total weight $0 \pmod{p}$ that cannot be similarly partitioned, then $[\id, w]_{\lR} \neq [\id, w]_{\cdf}$.  It will be more convenient at first to phrase the first possibility in terms of reflection length, allowing us to state a uniform result for all complex reflection groups.

\begin{prop}\label{prop: first condition}
    Let $W$ be any complex reflection group and let $w$ be an element of $W$.
    If $\lr(w) > \codimfix(w)$, then there exists a reflection that belongs to $[\id, w]_{\lR}$ but not to $[\id, w]_{\cdf}$.
\end{prop}

\begin{proof}
We first establish the result for irreducible groups by a case-based approach; then at the end we show that it extends to reducible groups.

Consider $w \in G(m, p, n)$ with $\lR(w) > \codimfix(w)$.  By \Cref{prop: shi}, we equivalently have that in any maximum null cycle partition of $w$, there is either a part with at least three cycles, or there is a part with at least two cycles and nonzero total weight.  Fix such a maximum null cycle partition $\Pi$, let $B$ be the part promised by the last sentence, and let $C_1 = (a \cdots)$ and $C_2 = (b \cdots)$ be two of the cycles in $B$, of respective nonzero weights $\alpha$ and $\beta$.  Choose a reflection $t := [(a \ b); \mathbf{0}] \in G(m, p, n)$ that transposes an element from one of these cycles with an element from the other.  Since $t$ is a reflection, $\lr(t) = \codimfix(t) = 1$.  Since the entries transposed by $t$ are in different cycles of $w$, multiplying $w$ by $t = t^{-1}$ merges these two cycles into a single cycle, necessarily of weight $\alpha + \beta$.  Let us consider $\codimfix(tw)$ and $\lr(tw)$.

If $|B| = 2$, then by assumption $\alpha + \beta \neq 0$.  If instead $|B| > 2$, since $B$ is a part of a maximum null cycle partition, no subset of $B$ can have weights that sum to $0$ (or else we could split this subset into its own part, increasing the value of the partition), so also in this case $\alpha + \beta \neq 0$.  Consequently in both cases $tw$ has the same number of weight-$0$ cycles as $w$.  It follows that $\codimfix(tw) = \codimfix(w)$.  Thus $t \not \in [\id, w]_{\cdf}$.

Let $\Pi'$ be the partition of the cycles of $tw$ that we get from $\Pi$ by deleting the two cycles $C_1$ and $C_2$ from $B$ and replacing them with the merged cycle $t C_1 C_2$, otherwise leaving the partition the same.  Clearly $v(\Pi') = v(\Pi)$.  Thus $\lR(tw) \leq n + c(tw) - v(\Pi') = n + c(w) - 1 - v(\Pi) = \lR(w) - 1$.  On the other hand, by subadditivity, $1 + \lR(tw) = \lR(t) + \lR(tw) \geq \lR(w)$, so in fact we have equality, and $t \in [\id, w]_{\lR}$.  This completes the proof in the case of the infinite family $G(m, p, n)$.

We next consider the exceptional groups.  Here we proceed by a brute-force computer calculation.  For each exceptional group $W$, we performed for each conjugacy class representative $w$ the following computation: if $\lR(w) > \codimfix(w)$, we checked for each reflection $t$ whether $t \leq_{\lR} w$ and $t \not\leq_{\cdf} w$.  In all cases, the result of the calculation was to verify the existence of such a reflection.  The entire computation took a few minutes runtime on CoCalc, using SageMath and its interface with GAP and the CHEVIE package \cite{ GAP, chevie,sagemath}.

Finally, we extend the result to all (not necessarily irreducible) complex reflection groups.  Let $W = W_1 \times \cdots \times W_k$ be the decomposition of $W$ into irreducibles (an internal direct product) and $w = w_1 \cdots w_k$ the corresponding decomposition of $w$.  Reflection length and fixed space codimension are both additive over direct products, so if $\lr(w) > \codimfix(w)$, then some $w_i$ must satisfy $\lR(w_i) > \codimfix(w_i)$.   Since $W_i$ is irreducible, it falls into one of the cases above, and so there exsits a reflection $t$ in $W_i$ that satisfies $t \in [\id, w_i]_{\lR}$ and $t \not\in [\id, w_i]_{\cdf}$.  The set of reflections of $W$ is precisely the union of the sets of reflections of the $W_i$, so $t$ is a reflection in $W$.  Using again that reflection length and fixed space codimension are additive over direct products, it follows immediately that $t \leq_{\lR} w$ but $t \not\leq_{\cdf} w$.  This completes the proof.
\end{proof}

The following observation has appeared many times in the literature (e.g., as \cite[Lem.~2.4]{FosterGreenwood}); we include its proof for completeness.

\begin{prop}\label{prop:inclusion}
    Let $W$ be any reflection group and $w\in W$. If $\lr(w)=\codimfix(w)$, then $[\id, w]_{\lr}\subseteq [\id, w]_{\cdf}$.  Moreover, in this case $\lR(u) = \codimfix(u)$ for all $u \in [\id, w]_{\lR}$.
\end{prop}

\begin{proof}
For any $u\in [\id, w]_{\lr}$, we have by \eqref{eq:codim < lR} and the definition of the reflection length order that $\codimfix(u)+\codimfix(u^{-1}w)\leq \lr(u)+\lr(u^{-1}w)=\lr(w)$. By the subadditivity of $\codimfix(-)$, we have $\codimfix(w)\leq \codimfix(u)+\codimfix(u^{-1}w)$. Since $\lr(w)=\codimfix(w)$, equality is forced in all inequalities, and so we have both $\lR(u) = \codimfix(u)$ and $\codimfix(u)+\codimfix(u^{-1}w)=\codimfix(w)$, i.e., $u\in [\id, w]_{\cdf}$.
\end{proof}

(In fact the preceding proof is valid for any two functions that satisfy the hypotheses of \Cref{prop:poset} and \eqref{eq:codim < lR}.)  We now shift our focus to the particular case of the combinatorial groups $G(m, p, n)$, and show the necessity of the second condition in~\Cref{thm:main}.

\begin{prop}\label{prop: second condition}
    Let $W = G(m, p, n)$ and $w \in W$.  Suppose that $\lR(w) = \codimfix(w)$ and that there exists a subset of the cycles of $w$ whose weight is $0 \pmod{p}$ but which cannot be partitioned into pairs of cycles whose weights sum to $0$ and singleton sets containing a cycle of weight $0 \pmod{p}$.  Then there is an element of $[\id, w]_{\cdf}$ that does not belong to $[\id, w]_{\lR}$.
\end{prop}

\begin{proof}
Fix $w\in G(m,p,n)$ with cycles $C_1, \ldots, C_k$. Let $S=\{ C_{i_1}, \ldots, C_{i_s}\}$ be a subset of cycles of $w$ whose weights sum to $0\pmod{p}$, but which cannot be partitioned into pairs of cycles whose weights sum to $0$ and singleton sets containing a cycle of weight $0\pmod{p}$. Since any cycles of weight $0 \pmod{p}$ can be removed from $S$ while preserving this property, it suffices to assume that $S$ does not contain any such cycle of $w$. 

Let $u = w|_S$ be the element of $G(m, 1, n)$ that agrees with $w$ on $\supp(S)$ and acts as the identity (with weight $0$) on $\{1, \ldots, n\} \smallsetminus \supp(S)$.  Since the sum of the weights of the cycles in $S$ is $0 \pmod{p}$, in fact $u \in G(m, p, n)$.  We now show that $u \in [\id, w]_{\cdf}$. 
Since $S$ contains no cycles of weight $0 \pmod{p}$, the only cycles of weight $0$ in $u$ are the fixed points outside $\supp(S)$, and so $c_0(u) = n - |\supp(S)|$.  On the other hand, $u^{-1}w$ contains a fixed point (of weight $0$) for each element of $\supp(S)$ and also all the cycles of $w$ not in $S$; in particular, it contains all weight-$0$ cycles of $w$.  Thus $c_0(u^{-1}w) = |\supp(S)| + c_0(w)$.  It follows immediately that 
\begin{align*}
\codimfix(u)+\codimfix(u^{-1}w)
&=(n-c_0(u)) + (n-c_0(u^{-1}w))\\
&= |\supp(S)| + n-|\supp(S)| - c_0(w)\\
&=n-c_0(w)\\
&= \codimfix(w).
\end{align*}
Then $u\in [\id, w]_{\cdf}$.

On the other hand, by~\eqref{eq:codim < lR}, \Cref{cor:shi}, and the defining property of $S$, we have $\lr(u)>\codimfix(u)$. Thus by \Cref{prop:inclusion}, 
$u\notin [\id, w]_{\lr}$.  This completes the proof.
\end{proof}

\begin{cor}
    \label{cor:necessary}
    If $w \in G(m, p, n)$ satisfies $[\id, w]_{\lr} = [\id, w]_{\cdf}$, then $\lR(w) = \codimfix(w)$
\textup{and} any subset of cycle weights that sums to $0 \pmod{p}$ is a disjoint union of some weights that are $0 \pmod{p}$ and some pairs of weights that sum to $0$.
\end{cor}

\begin{proof}
  The result follows immediately from~\eqref{eq:codim < lR}, \Cref{prop: first condition}, and~\Cref{prop: second condition}.
\end{proof}

\section{Sufficiency}\label{sec:sufficient}

It remains to show that the given conditions are sufficient to imply $[\id, w]_{\lr} = [\id, w]_{\cdf}$.  We begin with a general lemma on the structure of the $\cdf$-interval below an arbitrary element $w$, whose statement requires an additional definition.

\begin{defn}
Suppose that $u, w \in G(m, p, n)$.  Define an equivalence relation $\sim$ on the cycles of $w$, as follows: for two cycles $C_1$ and $C_2$ of $w$, we have that $C_1 \sim C_2$ if there exists a cycle of $u$ that intersects both $C_1$ and $C_2$, and we extend by transitivity.  Denote by $\Pi_u(w)$ the resulting set partition of the cycles of $w$.  
\end{defn}

\begin{ex}\label{ex:partition}
\begin{enumerate}
    \item For any $w \in G(m, p, n)$, we have $\Pi_\id(w)$ is the fully refined partition in which each cycle belongs to its own part.
    \item If $c$ is an element of $G(m, p, n)$ that has a single $n$-cycle, then $\Pi_c(w)$ is the trivial partition in which all cycles of $w$ belong to the same part.
    \item For $
    w = [(1 \ 2 \ 3)(6 \ 7); (0, 0, 1, -1, -2, 2, 0, 0)] 
    $
and $u = [(2 \ 3)(5 \ 6)(7 \ 8); (0, 0, 0, 3, 0, 3, 0, 0)] $ in $G( 6, 6, 8)$, we have 
    \[
    \Pi_u(w) = [(1 \ 2 \ 3) \mid (4) \mid (5)(6 \ 7)(8) ].  \]
    \end{enumerate}
\end{ex}

For each part $B$ of $\Pi_u(w)$, we have by construction that the underlying set of $B$ is stabilized by (the underlying permutations of) both $u$ and $w$.  Let $u|_B$ and $w|_B$ be the associated restrictions, both of which may be viewed as elements of $G(m, 1, \#\supp(B)) \subseteq G(m, 1, n)$. 

\begin{ex}
    Taking $u, w$ as in Example~\ref{ex:partition}.3, we have
    \begin{align*}
    w|_{B_1} & = [(1\ 2 \ 3);(0, 0, 1, 0, 0, 0, 0, 0)], & u|_{B_1}& =[(2 \ 3); \mathbf{0}], \\
    w|_{B_2} & =[\id;(0,0,0,-1,0,0,0,0)] & u|_{B_2} & = [\id; (0, 0, 0, 3, 0, 0, 0, 0)], \\
    w|_{B_3} & =[(6 \ 7); (0, 0,0,0,-2,2, 0, 0)], & u|_{B_3} & = [(5 \ 6)(7 \ 8); (0, 0, 0, 0, 0, 3, 0, 0)].
    \end{align*}
The elements $w|_{B_1}$ and $u|_{B_1}$ may be viewed as elements of $G(6, 1, 3) \subset G(6, 1, 8)$ via the embedding in which the first three coordinates are permuted.  Note that $w|_{B_1}$ is \emph{not} an element of $G(6, 6, 8)$, since its weight is $1$ (not $0$), even though $w$ is.
\end{ex}

\begin{prop}\label{prop:general bound}
Suppose that $u, w$ are any two elements of $G(m, p, n)$.  Then
    \[
    \codimfix(u) + \codimfix(u^{-1}w) \geq n + c(w) - 2|\Pi_u(w)| + \#\{\text{parts of $\Pi_u(w)$ of nonzero weight}\}.
    \]
Furthermore, in the case of equality, for each part $B$ of $\Pi_u(w)$ we have that $c(u|_B) + c(u^{-1}w|_B) = \#\supp(B) - c(w|_B) + 2$ and exactly one of the following holds:
\begin{itemize}
   \item $\wt(w|_B) = 0$ and all cycles of $u|_B$ and $u^{-1}w|_B$ have weight $0$, 

   \item $\wt(u|_B) = \wt(w|_B) \neq 0$, one cycle of $u|_B$ has weight $\wt(w|_B)$ while the others have weight $0$, and all cycles of $u^{-1}w|_B$ have weight $0$, or

   \item $\wt(u^{-1}w|_B) = \wt(w|_B) \neq 0$, one cycle of $u^{-1}w|_B$ has weight $\wt(w|_B)$ while the others have weight $0$, and all cycles of $u|_B$ have weight $0$.
\end{itemize}
\end{prop}
\begin{proof}
Choose a part $B$ of $\Pi_u(w)$.  Given any partition of $\supp(B)$ into two parts, it may be that there is a cycle of $w|_B$ that includes elements from both parts.  If not, there must be a cycle of $u|_B$ that includes elements from both parts (since otherwise the cycles of $w$ in $B$ would not be connected under $\sim$).  Thus, the underlying permutations of $u|_B$ and $w|_B$ generate a group that acts transitively on $\supp(B)$. 
Furthermore, since $u$ and $w$ stabilize $\supp(B)$, we have $(u|_B)^{-1} \cdot (w|_B) = (u^{-1}w)|_B$, with 
$u|_B$ and $u^{-1}w|_B$ generating the same group as 
$u|_B$ and $w|_B$. 
    
If $\tau, \sigma_1, \ldots, \sigma_k$ are permutations of $[N]$ such that $\sigma_1 \cdots \sigma_k = \tau$ and the group generated by $\sigma_1, \ldots, \sigma_k$ acts transitively on $[N]$, then \cite[Eq.~(4)]{BMS} 
    \begin{equation}
    \label{eq:transitive length bound}
    \sum_{i = 1}^k c(\sigma_i) \leq (k - 1)N - c(\tau) + 2.    
    \end{equation}
Taking $k = 2$, $N = \#\supp(B)$, and $\tau, \sigma_1, \sigma_2$ to be the underlying permutations of $w|_B, u|_B, u^{-1}w|_B$, respectively, it follows immediately from \eqref{eq:transitive length bound} that
    \begin{equation}
    \label{eq:part length bound}
    c(u|_B) + c(u^{-1}w|_B) \leq \#\supp(B) - c(w|_B) + 2.
    \end{equation}
Summing \eqref{eq:part length bound} over all parts $B$ of $\Pi_u(w)$, we have that
    \begin{equation}
    \label{eq:cycle bound}
    c(u) + c(u^{-1}w) \leq n - c(w) + 2 |\Pi_u(w)|.
    \end{equation}

Let us now consider how many cycles of $u$ and $u^{-1}w$ may be weight $0$.  Obviously $c_0(u|_B) + c_0(u^{-1}w|_B) \leq c(u|_B) + c(u^{-1}w|_B)$ (since each weight-$0$ cycle is a cycle).  Furthermore, if the total weight of $w|_B$ is nonzero, then (since $u|_B \cdot u^{-1}w|_B = w|_B$) at least one of $u|_B$, $u^{-1}w|_B$ must have nonzero weight, and so in this case at least one cycle of at least one of the two factors must have nonzero weight.  Therefore we can refine the previous inequality to give
    \begin{equation}
    \label{eq:considering weights}
    c_0(u|_B) + c_0(u^{-1}w|_B) \leq c(u|_B) + c(u^{-1}w|_B) - \Big[ \wt(w|_B) \neq 0 \Big],       
    \end{equation}
where the last summand on the right is an Iverson bracket.  Summing \eqref{eq:considering weights} over all parts $B$, we have
    \[
    c_0(u) + c_0(u^{-1}w) \leq c(u) + c(u^{-1}w) - \#\{\text{parts of $\Pi_u(w)$ of nonzero weight}\}.
    \]
Combining this with \eqref{eq:cycle bound} yields
    \[
    c_0(u) + c_0(u^{-1}w) \leq n - c(w) + 2|\Pi_u(w)| - \#\{\text{parts of $\Pi_u(w)$ of nonzero weight}\}.
    \]
Subtracting both sides from $2n$ gives the desired inequality.  The equality condition forces equality in \eqref{eq:part length bound} and \eqref{eq:considering weights} for every part $B$; 
when $\wt(w) = 0$, this forces $c_0(u|_B) = c(u|_B)$ and $c_0(u^{-1}w|_B) = c(u^{-1}w|_B)$, while otherwise it forces one of $u$ and $u^{-1}w$ to have all cycles of weight $0$ and the other to have all but one cycle of weight $0$ (with the extra cycle necessarily of weight $\wt(w)$).
\end{proof}

\begin{prop}
    \label{prop:partition for interval}
    Suppose that $w \in G(m, p, n)$ is arbitrary and $u \in [\id, w]_{\cdf}$.  Then 
    \begin{enumerate}[(1)]
        \item $\codimfix(u) + \codimfix(u^{-1}w) = n + c(w) - 2|\Pi_u(w)| + \#\{\text{parts of $\Pi_u(w)$ of nonzero weight}\}$ and
        \item every part of $\Pi_u(w)$ is either a singleton or consists of exactly two cycles of nonzero weights that sum to $0$.
    \end{enumerate}    
\end{prop}
\begin{proof}
Since $u \in [\id, w]_{\cdf}$, we have $\codimfix(u) + \codimfix(u^{-1}w) = \codimfix(w) = n - c_0(w)$.  By Proposition~\ref{prop:general bound}, we have
    \begin{multline}
    \label{eq:using the proposition}
    n - c_0(w) = \codimfix(u) + \codimfix(u^{-1}w) \geq \\ n + c(w) - 2|\Pi_u(w)| + \#\{\text{parts of $\Pi_u(w)$ of nonzero weight}\},    
    \end{multline}
and hence that
    \begin{equation}
    \label{eq:w cycle upper bound}
    c_0(w) + c(w)  \leq |\Pi_u(w)| + \#\{\text{parts of $\Pi_u(w)$ of weight $0$}\}.
    \end{equation}
Let $a$ be the number of parts of $\Pi_u(w)$ consisting of a single weight-$0$ cycle of $w$, let $b$ be the number of parts consisting of a single cycle of nonzero weight, let $c$ be the number of parts having total weight $0$ with more than one cycle, and let $d$ be the total number of remaining parts.  
Then of course $|\Pi_u(w)| = a + b + c + d$,  $c_0(w) \geq a$, and $c(w) \geq a + b + 2c + 2d$.  Plugging these inequalities into \eqref{eq:w cycle upper bound} yields
\begin{align*}
    2a + b + 2c + 2d & = a + (a + b + 2c + 2d) \\
    & \leq c_0(w) + c(w) \\
    & \leq (a + b + c + d) + a + c \\
    & = 2a + b + 2c + d.
\end{align*}
This immediately forces $d = 0$.  Moreover, when $d = 0$, the equality between the first and last terms forces equality in the middle: hence we have equality in \eqref{eq:using the proposition} (the first of the two desired statements) as well as $c_0(w) = a$ and $c(w) = a + b + 2c$.  From $c_0(w) = a$, we learn that each weight-$0$ cycle in $w$ belongs to its own part in $\Pi_u(w)$.  From $c(w) = a + b + 2c$, we learn that each non-singleton part in $\Pi_u(w)$ consists of exactly two cycles and has total weight $0$; since the weight-$0$ cycles are in singleton parts, it follows that the non-singleton part of $\Pi_u(w)$ contain two cycles, both of which have nonzero weight, and that the two weights sum to $0$.  This is precisely the second claim.
\end{proof}

As a last preliminary step towards Theorem~\ref{thm:main}, we prove the theorem in the case of certain very special elements.
\begin{lem}
\label{lem:special elements}
     Suppose that $w \in G(m, p, n)$ 
    is an element with a single cycle, or with two cycles of nonzero weights that sum to $0$.  Then $[\id, w]_{\cdf} = [\id, w]_{\lr}$. 
\end{lem}
\begin{proof}
    We consider three cases.
    
    First, suppose that $w$ consists of a single cycle of weight $0$.  By~\Cref{thm: reflection length} and the definition of $\codimfix(-)$,
    $\lR(w) = \codimfix(w) = n - 1$, and therefore by~\Cref{prop:inclusion},
    $[\id, w]_{\lr} \subseteq [\id, w]_{\cdf}$.  We consider the reverse inclusion.  Let $u \in [\id, w]_{\cdf}$.  Since $w$ has only one cycle, $\Pi_u(w)$ is the unique set partition of the one-element set containing this cycle.  By the definition of $\leq_{\cdf}$, we have \begin{align*} 
    \codimfix(u) + \codimfix(u^{-1}w) 
    & = n + 1 - 2 + 0 \\
    & = n + c(w) - 2|\Pi_u(w)| + \#\{\text{parts of $\Pi_u(w)$ of nonzero weight}\}.
    \end{align*}
    Therefore, by Proposition~\ref{prop:general bound}, all cycles of $u$ and $u^{-1}w$ have weight $0$.  It follows immediately from Proposition~\ref{prop: shi} that $\lR(u) = \codimfix(u)$ and $\lR(u^{-1}w) = \codimfix(u^{-1}w)$, and hence that $u \in [\id, w]_{\lr}$, as claimed.

    Second, suppose that $w$ consists of a single cycle of nonzero weight.  By~\Cref{thm: reflection length} and the definition of $\codimfix(-)$, $\lR(w) = \codimfix(w) = n$, and therefore by~\Cref{prop:inclusion}, $[\id, w]_{\lr} \subseteq [\id, w]_{\cdf}$.  We consider the reverse inclusion.  Let $u \in [\id, w]_{\cdf}$.  Since $w$ has only one cycle, $\Pi_u(w)$ is the unique set partition of the one-element set containing this cycle.  By the definition of $\leq_{\cdf}$, we have \begin{align*} 
    \codimfix(u) + \codimfix(u^{-1}w) 
    & = n + 1 - 2 + 1 \\
    & = n + c(w) - 2|\Pi_u(w)| + \#\{\text{parts of $\Pi_u(w)$ of nonzero weight}\}.
    \end{align*}
    Therefore, by Proposition~\ref{prop:general bound}, one of $u$ and $u^{-1}w$ has one cycle of nonzero weight (and some number of cycles of weight $0$) and the other has all cycles of weight $0$; moreover, this nonzero weight is $\wt(w)$, which is $0 \pmod{p}$.  It follows immediately from Proposition~\ref{prop: shi} that $\lR(u) = \codimfix(u)$ and $\lR(u^{-1}w) = \codimfix(u^{-1}w)$, and hence that $u \in [\id, w]_{\lr}$, as claimed.

    Finally, suppose that $w$ consists of two cycles of nonzero weights that sum to $0$.  By~\Cref{thm: reflection length} and the definition of $\codimfix(-)$, $\lR(w) = \codimfix(w) = n$, and therefore by~\Cref{prop:inclusion}, $[\id, w]_{\lr} \subseteq [\id, w]_{\cdf}$.  We consider the reverse inclusion.  Let $u \in [\id, w]_{\cdf}$. We have two sub-cases, depending on the structure of $\Pi_u(w)$.

    First, suppose that $\Pi_{u}(w)$ is the set partition with a single part that contains both cycles of $w$. By the definition of $\leq_{\cdf}$, we have \begin{align*}     \codimfix(u) + \codimfix(u^{-1}w) 
    & = n + 2 - 2 + 0 \\
    & = n + c(w) - 2|\Pi_u(w)| + \#\{\text{parts of $\Pi_u(w)$ of nonzero weight}\}.
    \end{align*}
    Therefore, by Proposition~\ref{prop:general bound}, all cycles of $u$ and $u^{-1}w$ have weight $0$. It follows immediately from Corolary~\ref{cor:shi} that $\lR(u) = \codimfix(u)$ and $\lR(u^{-1}w) = \codimfix(u^{-1}w)$, and hence that $u \in [\id, w]_{\lr}$, as claimed.
    
    Alternatively, it could be that $\Pi_u(w)$ is the set partition of two parts $B_1, B_2$, each containing exactly one cycle of $w$. By the definition of $\leq_{\cdf}$, we have \begin{align*} 
    \codimfix(u) + \codimfix(u^{-1}w) 
    & = n + 2 - 4 + 2 \\
    & = n + c(w) - 2|\Pi_u(w)| + \#\{\text{parts of $\Pi_u(w)$ of nonzero weight}\}.
    \end{align*}
    By Proposition~\ref{prop:general bound}, 
    we know that for either $i = 1, 2$, either $u|_{B_i}$ has one cycle with weight $\wt(w|_{B_i})$ and some cycles of weight $0$ (possibly none) while $u^{-1}w|_{B_i}$ has all cycles of weight $0$, or the reverse. 
    This leads to three possibilities for $u$, $u^{-1}w$: each of them may have all cycles of weight $0$, may have two cycles of nonzero weight that sum to zero and all other cycles of weight $0$, or may have one cycle of nonzero weight and all others of weight $0$ (in which case necessarily the one cycle must have weight $0 \pmod{p}$, since $u, u^{-1}w \in G(m, p, n)$).  In all three situations, we have immediately from~\Cref{cor:shi} that $\lr(u)=\codimfix(u)$ and $\lr(u^{-1}w)=\codimfix(u^{-1}w)$, and hence that $u\in[\id, w]_{\lr}$. 

    Since the preceding cases are exhaustive, the proof is complete.
\end{proof}

We are ready now to complete the second direction of the biconditional in the main theorem.  

\begin{prop}
    \label{prop:sufficient}
    If $w \in G(m, p, n)$ satisfies $\lR(w) = \codimfix(w)$ and any subset of cycle weights that sums to $0 \pmod{p}$ is a disjoint union of some weights that are $0 \pmod{p}$ and some pairs of weights that sum to $0$, then $[\id, w]_{\lr} = [\id, w]_{\cdf}$.
\end{prop}

\begin{proof}
Let $w \in G(m, p, n)$ be as in the statement.  Since $\lR(w) = \codimfix(w)$, we have by \Cref{prop:inclusion} that $[\id, w]_{\lR} \subseteq [\id, w]_{\cdf}$, and we seek to prove the opposite inclusion.  To that end, choose $u \in [\id, w]_{\cdf}$.

Consider the partition $\Pi_u(w) = \{B_1, \ldots, B_k\}$ of the cycles of $w$ induced by $u$.  Since $u \in [\id, w]_{\cdf}$, we have on one hand by Proposition~\ref{prop:partition for interval}(2) that each $B_i$ either consists of a single cycle of $w$ or of a pair of cycles whose weights sum to $0$.  On the other hand, we have by Proposition~\ref{prop:partition for interval}(1) that $\codimfix(u) + \codimfix(u^{-1}w) = n + c(w) - 2|\Pi_u(w)| + \#\{\text{parts of $\Pi_u(w)$ of nonzero weight}\}$.  Therefore, by the equality case of Proposition~\ref{prop:general bound}, we have for each part $B_i$ of $\Pi_u(w)$ that either $u|_{B_i}$ has all cycles of weight $0$, or that $u|_{B_i}$ has a single cycle whose weight is equal to $\wt(w|_{B_i})$ and some number of other cycles of weight $0$.
Combining these two separate statements leaves us with three possibilities for each part $B_i$: either
\begin{itemize}
    \item $w|_{B_i}$ consists of two cycles of nonzero weight whose weights sum to $0$, and $u|_{B_i}$ consists of a number of cycles of weight $0$, or
    \item $w|_{B_i}$ consists of a single cycle, and $u|_{B_i}$ consists of a number of cycles of weight $0$, or
    \item $w|_{B_i}$ consists of a single cycle of nonzero weight, and $u|_{B_i}$ has one cycle of this weight and possibly some other cycles, all of which have weight $0$.
\end{itemize}
By taking the union over all $B_i$, it follows in particular that the multiset of cycle weights of $u$ is the union of a submultiset of the multiset of cycle weights of $w$ and a multiset containing some number of copies of $0$.  

Let $S$ be the set consisting of all cycles $C$ of $w$ that fall into the third category above, i.e., for which there is a corresponding cycle of $u$ supported on $\supp(C)$ having the same nonzero weight as $C$.  Thus, the multiset of nonzero cycle weights of $u$ is precisely the same as the multiset of weights of cycles in $S$.  Since $u \in G(m, p, n)$, the nonzero cycle weights of $u$ sum to $0 \pmod{p}$, and therefore $S$ is a set of cycles of $w$ with the property that the sum of the weights of its cycles is $0 \pmod{p}$.  By the hypothesis on $w$, it follows that there exists a partition (call it $Q$) of $S$ consisting of some singleton sets containing a cycle of weight $0 \pmod{p}$ and some pairs in which the weights are nonzero but sum to $0$ (and no other parts).

Define a partition $P_u(w)$ of the cycles of $w$ into parts of size $1$ and $2$, as follows: 
\begin{itemize}
    \item First, if two cycles of $w$ belong to the same part $B_i$ of $\Pi_u(w)$, then they belong to the same part in $P_u(w)$.  (The cycles in each pair in this case have nonzero weights that sum to $0$.)
    \item Second, if two cycles of $w$ both belong to $S$ (so, by definition of $S$, are not covered in the prior case) and belong to the same part in $Q$, then they belong to the same part in $P_u(w)$.  (The cycles in each pair in this case have nonzero weights that sum to $0$.)
    \item Finally, let $S'$ be the set of cycles of $w$ not covered in either of the prior cases.  Since $w \in G(m, p, n)$, the sum of all its cycle weights is $0\pmod{p}$.  The cycles of $w$ covered by the preceding cases have total weight $0$, so the cycles of $S'$ have total weight $0\pmod{p}$.  Therefore, by the hypothesis on $w$, there is a partition $Q'$ of $S'$ each of whose parts is either a singleton of weight $0\pmod{p}$ or a pair of cycles of nonzero weight that sum to $0$.  Let $P_u(w)$ include all the parts of $Q'$.
\end{itemize}
An example of this construction is given as Example~\ref{ex:P_u(w)} below. 

Since each part of $P_u(w)$ is a union of parts of $\Pi_u(w)$, the underlying permutations of $u$ and $w$ respect the underlying set partition of $[n]$, and it makes sense to speak of the restrictions $u|_B$ and $w|_B$ for $B$ a part of $P_u(w)$.  Furthermore, the partition $P_u(w)$ is a \emph{null} cycle partition of $w$: by construction, each part of size two has total weight $0$, and each part of size one has weight $0 \pmod{p}$.  Thus, each element $w|_B$ belongs to $G(m, p, n)$, not just to $G(m, 1, n)$.  Moreover, the same holds true for $u$ because, by construction, each restriction $u|_B$ either has weight $0$ or has the same weight as $w|_B$.  Thus, we may view the products
\[
w = \prod_{B \in P_u(w)} w|_B
\qquad \text{and} \qquad
u = \prod_{B \in P_u(w)} u|_B
\]
as factorizations of $w, u$ into elements of $G(m, p, n)$.

Next we claim that for each part $B$ of $P_u(w)$, the element $u|_B$ belongs to $\big[\id, w|_B\big]_{\cdf}$. 
To see this we must consider the several cases in which parts $P_u(w)$ were constructed.  If $B$ is a part of $\Pi_u(w)$ (consisting of either one or two cycles of $w$), then we combine Propositions~\ref{prop:general bound} and~\ref{prop:partition for interval} in the same way as before: \Cref{prop:partition for interval} says that we are in the equality case of \Cref{prop:general bound}, and the conditions of the equality case imply that $\codimfix(u|_B) + \codimfix(u^{-1}w|_B) = \codimfix(w|_B)$ in each case.  The other possibility is that $B$ consists of two cycles, each of which formed a singleton part in $\Pi_u(w)$.  This could happen either because the two cycles of $w$ belonged to the same part in the partition $Q$ of $S$, or because they belonged to the same part in the partition $Q'$ of $S'$.  We spell out the details for only the first of these two cases; the other is very similar.  So suppose that $B = \{C_1, C_2\}$ consists of two cycles of $w$ having nonzero weights that sum to $0$, and that $u|_B$ consists of some number of cycles of weight $0$, as well as one cycle of the same weight as $C_1$ and another of the same weight as $C_2$.  As before, by Propositions~\ref{prop:general bound} and~\ref{prop:partition for interval}, we have for $i = 1, 2$ that $c(u|_{\{C_i\}}) + c(u^{-1}w|_{\{C_i\}}) = \#\supp(C_i) + 1$ and that $c_0(u|_{\{C_i\}}) = c(u|_{\{C_i\}}) - 1$, $c_0(u^{-1}w|_{\{C_i\}}) = c(u^{-1}w|_{\{C_i\}})$.  Substituting the latter two equations into the former, subtracting from $2\#\supp(C_i)$, and adding the two resulting equations for $i = 1, 2$ gives $\codimfix(u|_B) + \codimfix(u^{-1}w|_B) = \#\supp(B) = \codimfix(w|_B)$, as needed.

Since each element $w|_B$ is either a single cycle of weight $0 \pmod{p}$ or a product of two cycles of nonzero weights that sum to $0$, we have by Lemma~\ref{lem:special elements} that $u|_B \in [\id, w|_B]_{\lr}$.  The final result is now a simple calculation: we have
\begin{multline*}
\codimfix(w) = \codimfix(u) + \codimfix(u^{-1}w) \leq \\
    \lR(u) + \lR(u^{-1}w) 
    \overset{\substack{\text{sub-} \\ \text{additivity} \\ \text{of } \lR}}{\leq} \sum_{B \in P_u(w)} \lR(u|_B) + \sum_{B \in P_u(w)} \lR(u^{-1}w|_B) = \\
    \sum_{B \in P_u(w)} \left(\lR(u|_B) +  \lR(u^{-1}w|_B) \right)
    \overset{\text{Lem.~\ref{lem:special elements}}}{=} \sum_{B \in P_u(w)} \codimfix(w|_B)
    \overset{\substack{\text{cycle}\\\text{partition}}}{=} \codimfix(w),
\end{multline*}
and therefore $\lR(u) = \codimfix(u)$ and $\lR(u^{-1}w) = \codimfix(u^{-1}w)$, so $u \in [\id, w]_{\lr}$, as claimed.
\end{proof}

\begin{ex}\label{ex:P_u(w)}
Consider the elements 
    \[
    w = [(1 \ 2)(3)(4)(5)(6)(7)(8)(9)(10); (0, 1, -1, 2, -2, 4, -4, 8, 8, 8 )] \in G( 16, 8, 10)
    \]
and $u = [(1)(2)(3)(4 \ 5)(6)(7)(8 \ 9)(10); (0, 0, 0, 1, -1, 4, -4, 0, 0, 8)] $.  Then 
    \[
    \Pi_u(w) = [(1 \ 2) \mid (3) \mid (4)(5) \mid (6) \mid (7) \mid (8)(9) \mid (10)].
    \]
Observe that this partition of the cycles is not a \emph{null} cycle partition because in the part $B_1 = \{(1 \ 2)\}$, the cycle weight is not $0\pmod{8}$. In the bulleted list in the definition of $P_u(w)$, the first category consists of the parts $B_3 = \{(4), (5)\}$ and $B_6 = \{(8), (9)\}$, the second category consists of the parts $B_1 = \{(1 \ 2)\}$ and $B_2 = \{(3)\}$, and the third category consists of the parts $B_4 = \{(6)\}$, $B_5 = \{(7)\}$, and $B_7 = \{(10)\}$.  Therefore $S = \{(6), (7), (10)\}$.  The desired partition $Q$ of $S$ is $[(6)(7) \mid (10)]$ (in this case it is unique).  To define $P_u(w)$, we have three steps.  The first step creates parts $(4)(5)$ and $(8)(9)$.  The second step creates a part $(6)(7)$.  In the third step, we have leftover cycles $S' = \{(1 \ 2), (3), (10)\}$, and the desired partition $Q'$ is $[(1 \ 2)(3) \mid (10)]$.  Thus $P_u(w) = [(1 \ 2)(3) \mid (4)(5) \mid (6)(7) \mid (8)(9) \mid (10)]$.  

In contrast to $\Pi_u(w)$, this partition \emph{is} a null cycle partition of $w$.  Furthermore the restriction of $u$ to any part has weight $0 \pmod{p}$ (including possibly weight $0$).
\end{ex}

The proof of the main theorem is essentially trivial at this point.

\maintheorem*

\begin{proof}
    The equivalence of the condition $\lR(w) = \codimfix(w)$ and the statement about cycle weights is~\Cref{cor:shi}.  With this equivalence in hand, the rest follows immediately from~\Cref{cor:necessary} and~\Cref{prop:sufficient}.
\end{proof}

\section{Further comments and questions}
\label{sec:further remarks}

We end with some remarks about the heritability of the property $[\id, w]_{\lR} = [\id, w]_{\cdf}$ and its comparison to other similar properties; about the extent to which special families of elements (regular elements and parabolic quasi-Coxeter elements) share this property in general complex reflection groups; and about which permutation statistics fall into the framework of Section~\ref{sec:general}.

\subsection{Heritability}

In \cite[Lem.~2.2]{FosterGreenwood}, a certain \emph{heritability} property was established for the relation $\lR(w) = \codimfix(w)$; namely, it was shown that if $\lR(x) = \codimfix(x)$ for every $x$ that is covered by $w$ in the $\cdf$-order, then also $\lR(w) = \codimfix(w)$.  The next result explores the extent to which similar results hold for the (stronger) condition $[\id, w]_{\lR} = [\id, w]_{\cdf}$.

\begin{prop}\label{prop:heritability}
    Fix a complex reflection group $W$.  For an element $w \in W$, the following are equivalent:
    \begin{enumerate}[(a)]
        \item $[\id, w]_{\lR} = [\id, w]_{\cdf}$
        \item $[\id, u]_{\lr}=[\id, u]_{\cdf}$ for every $u\in [\id, w]_{\lr}$
        \item $[\id, u]_{\lr}=[\id, u]_{\cdf}$ for every $u\in [\id, w]_{\cdf}$
        \item $\lR(w) = \codimfix(w)$ and $[\id, u]_{\lr}=[\id, u]_{\cdf}$ for every $u <_{\cdf} w$
        \item $\lR(u) = \codimfix(u)$ for every $u\in [\id, w]_{\cdf}$
\end{enumerate}
    The following conditions are implied by (but not equivalent to) those above:
    \begin{enumerate}[(a)]
    \addtocounter{enumi}{5}
        \item $\lR(w) = \codimfix(w)$ and $[\id, u]_{\lr}=[\id, u]_{\cdf}$ for every $u <_{\lR} w$
        \item $[\id, u]_{\lr}=[\id, u]_{\cdf}$ for every $u <_{\lR} w$
        \item $\lR(u) = \codimfix(u)$ for every $u\in [\id, w]_{\lr}$
        \item $\lR(u) = \codimfix(u)$ for every $u <_{\lR} w$
        \item $[\id, u]_{\lr}=[\id, u]_{\cdf}$ for every $u <_{\cdf} w$
    \item $\lR(u) = \codimfix(u)$ for every $u <_{\cdf} w$   
\end{enumerate}
    \end{prop}
\begin{proof}
Among the eleven conditions, the following implications are trivial (in each case the stronger condition explicitly includes the weaker condition):
(b)\Ra(a), 
(c)\Ra(a), 
(f)\Ra(g), (h)\Ra(i), and (d)\Ra(j).

By Equation~\eqref{eq:codim < lR} and Proposition~\ref{prop: first condition}, we have that $[\id, u]_{\lR} = [\id, u]_{\cdf}$ implies $\lR(u) = \codimfix(u)$ for any element $u$ of any complex reflection group.  From this fact, the following implications follow immediately:
(b)\Ra(f)\Ra(h), (c)\Ra(d)\Ra(e), (g)\Ra(i), and (j)\Ra(k).

Suppose $w$ satisfies (a), i.e., that $[\id, w]_{\lR} = [\id, w]_{\cdf}$, and choose any element $u$ of $[\id, w]_{\lR}$ (equivalently, any element of $[\id, w]_{\cdf}$).  By Proposition~\ref{prop:inclusion}, $[\id, u]_{\lR} \subseteq [\id, u]_{\cdf}$.  Choose $v \leq_{\cdf} u$.  By Proposition~\ref{prop:conjugacy}, $v^{-1}u \leq_{\cdf} u$, with $\codimfix(v^{-1}u) = \codimfix(u) - \codimfix(v) = \lR(u) - \lR(v)$.  By the transitivity of $\leq_{\cdf}$, we have $v^{-1}u \leq_{\cdf} w$ and $v \leq_{\cdf} w$, and so (by hypothesis) $v^{-1}u \leq_{\lR} w$ and $v \leq_{\lR} w$.  By Proposition~\ref{prop:inclusion}, this implies $\lR(v^{-1}u) = \codimfix(v^{-1}u)$ and $\lR(v) = \codimfix(v)$, and consequently $v \leq_{\lR} u$ as well.  Thus $[\id, u]_{\lR} = [\id, u]_{\cdf}$.  Since $u$ was arbitrary, we conclude that (a)\Ra(b) and (a)\Ra(c).

Suppose that condition (e) holds for $w$.  Certainly in this case $\lR(w) = \codimfix(w)$, so by Proposition~\ref{prop:inclusion} we have that $[\id, w]_{\lR} \subseteq [\id, w]_{\cdf}$.  Fix any $u \in [\id, w]_{\cdf}$.  By Proposition~\ref{prop:conjugacy}, since $\codimfix(-)$ is constant on conjugacy classes, $u^{-1}w \in [\id, w]_{\cdf}$ as well.  By assumption, this means $\codimfix(u) = \lR(u)$ and $\codimfix(u^{-1}w) = \lR(u^{-1}w)$.  Thus $\lR(u) + \lR(u^{-1}w) = \codimfix(u) + \codimfix(u^{-1}w) = \codimfix(w) = \lR(w)$, and so $u \in [\id, w]_{\lR}$.  Thus $[\id, w]_{\lR} = [\id, w]_{\cdf}$ in this case, i.e., we have shown that (e)\Ra(a).

Combining the preceding implications demonstrates that conditions (a)--(e) are equivalent, and that they imply all of (f)--(k).  To see that (f)--(k) are not equivalent to (a)--(e) in general, it is enough to observe first that the element
\[
w_1 = [\id; (1, 1, 1)] \in G(3, 3, 3)
\]
has $\lR(w_1) = 4 > 3 = \codimfix(w)$, so does not satisfy (d), and that $[\id, w_1]_{\cdf} = \{\id, w_1\}$, so $w_1$ satisfies (j) and (k); and second that the element
\[
w_2 = [\id; (1, 1, 1, -1, -1, -1)] \in G(3, 3, 6)
\]
does not satisfy (a) (by our main theorem) but that it satisfies $\lR(w_2) = 6 = \codimfix(w_2)$ (by Proposition~\ref{prop: shi}, using the partition whose three parts each contain one cycle of weight $1$ and one of weight $-1$) and satisfies (f) (a straightforward computer check) and hence also (by the above implications) (g), (h), and (i).
\end{proof}
\begin{remark}
    By Proposition~\ref{prop:inclusion}, condition (h) is equivalent to the condition $\lR(w) = \codimfix(w)$.
\end{remark}

\begin{remark}\label{rem:j = k}
    As observed in the proof of Proposition~\ref{prop:heritability}, for an element $w$ of a complex reflection group $W$ we have (j)\Ra(k).  In fact, the converse is also true, as we show now.

    Suppose that (k) holds for $w$, i.e., that $\lR(u) = \codimfix(u)$ for every $u <_{\cdf} w$. Choose any $u<_{\cdf} w$. For any $v\leq_{\cdf} u$, we have $v<_{\cdf} w$. Then $\lr(v)=\codimfix(v)$ by (k). Since $v\leq_{\cdf} u$ is arbitrary, we have (e) $\lR(v) = \codimfix(v)$ for every $v\in [\id, u]_{\cdf}$, which is equivalent to (a) $[\id, u]_{\lR} = [\id, u]_{\cdf}$. Since $u<_{\cdf} w$ is arbitrary, we have (j) $[\id, u]_{\lr}=[\id, u]_{\cdf}$ for every $u <_{\cdf} w$, as claimed.
\end{remark}

\begin{remark}
There are no additional relations among (f)--(k) beyond the relations (j)\,$\Leftrightarrow$\,(k), (f)\Ra(g)\Ra(i), and (f)\Ra(h)\Ra(i) established in the proof of Proposition~\ref{prop:heritability} and in Remark~\ref{rem:j = k}.  In particular:
\begin{itemize}
\item the element $[\id; (1,2,3,-1,-2,-3,3)]\in G(9,3,7)$ satisfies (h) and (i) but not (f) or (g) (the contradictory element is $u=[\id; (1,2,3,-1,-2,-3,0)]$);
\item the element $[\id; (1,1,1)]\in G(3,3,3)$ satisfies (g) and (i) but not (f) or (h) (since it has reflection length $4$ and fixed space codimension $3$);
\item the element $[\id; (1, 1, 1, 1)] \in G(4, 4, 4)$ satisfies (j) and (k) (because the only element below it in $\cdf$-order is the identity) but not (f), (g), (h), or (i) (a contradictory element is $u = [\id; (0,2,1,1)]$); and
\item the element $[\id; (1, 1, 1, -1, -1, -1)] \in G(3, 3, 6)$ satisfies (f), (g), (h), and (i) but not (j) or (k) (a contradictory element is $u = [\id; (1, 1, 1, 0, 0, 0)]$).
\end{itemize}
\end{remark}

\subsection{Specializing to nice families of elements}

As mentioned in the introduction, a major motivation of the present work was the result \cite[Cor.~6.6]{LM2021} that $[\id, c]_{\lR} = [\id, c]_{\cdf}$ when $c$ is a \emph{Coxeter element} in a well generated\footnote{That is, a complex reflection group $W$ that has a reflection generating set of size $\rank(W)$.  Every finite real reflection group is well generated;  in the infinite family, the well generated groups are the groups $G(m, 1, n)$ and $G(m, m, n)$.} complex reflection group $W$ (and in this case the intervals are the \emph{$W$-noncrossing partition lattice}).  In this section, we discuss two generalizations of this result, to larger families that naturally include the Coxeter elements.

\subsubsection{Regular elements}

One important family of elements in complex reflection groups are the \emph{regular} elements: an element $w \in W$ is regular if $w$ has an eigenvector $v$ that is not fixed by any reflection in $W$ (a so-called \emph{regular eigenvector}, whose associated eigenvalue is called a \emph{regular eigenvalue}).  (By Steinberg's theorem \cite[Thm.~4.7]{broue_book}, ``any reflection'' can equivalently be replaced by ``any non-identity element''.)
In a well generated complex reflection group $W$, the Coxeter elements may be characterized\footnote{There are several inequivalent definitions of Coxeter elements in the literature; see the cited paper of Reiner--Ripoll--Stump for details.} as the regular elements of maximum multiplicative order \cite[\S1.1]{RRS}, and it is natural to ask to what extent the equality $[\id, c]_{\lR} = [\id, c]_{\cdf}$ for Coxeter elements extends to other regular elements.  A first observation is that not all regular elements $w$ satisfy $\lR(w) = \codimfix(w)$; for example, the scalar matrix $[\id; (1, 1, 1)] \in G(3, 3, 3)$ is obviously regular but it has reflection length $4$.  In the infinite family $G(m, p, n)$, this is the only obstruction.

\begin{prop}\label{prop:regular}
    If $w$ is regular for $G(m, p, n)$ and $\lR(w) = \codimfix(w)$, then $[\id, w]_{\lR} = [\id, w]_{\cdf}$.
\end{prop}

As a prelude to the proof of Proposition~\ref{prop:regular}, we provide a combinatorial description of regular elements in the infinite family; this description is certainly not novel, but we were unable to find a fully explicit description in the literature.

\begin{lemma}\label{lem:regular elements}
    If $p < m$, an element $w$ is regular in $G(m, p, n)$ if and only if there is an integer $r \geq 1$ such that $w$ has $g:= \gcd(r, n)$ cycles, each of length $n/g$ and weight $r/g$.

    An element $w$ in $G(m, m, n)$ is regular if and only if there is an integer $r \geq 1$ such that $w$ has either (1) $g:= \gcd(r, n)$ cycles, each of length $n/g$ and weight $r/g$, or (2) $\gcd(r, n - 1) + 1$ cycles, of which all but one have length $(n - 1)/\gcd(r, n - 1)$ and weight $r/\gcd(r, n - 1)$, and the last is a $1$-cycle with weight $-r$.
\end{lemma}
\begin{proof}
    First consider the case $W = G(m, p, n)$ for $p < m$.  The set of reflecting hyperplanes of $W$ is the same as the set of reflecting hyperplanes of $G(m, 1, n)$ (there are fewer reflections, but the ``missing'' diagonal reflections all give redundant hyperplanes with the ones that exist in the smaller group).  Therefore, a vector $v \in \CC^n$ is regular for $W$ if and only if it is regular for $G(m, 1, n)$, and an element of $W$ is regular if and only if it is regular as an element of $G(m, 1, n)$.  Following \cite[\S4]{Garnier}, we have that $\epsilon = [(1 \ 2 \ \cdots \ n); (0, \ldots, 0, 1)]$ is a regular element in $G(m, 1, n)$ of order $mn$, with $\zeta_{mn} := \exp(2\pi i / mn)$ as a regular eigenvalue.  Therefore, for any power $r$, $\epsilon^r$ has $\zeta_{mn}^r$ as regular eigenvalue. For every divisor $k$ of $mn$, every $k$th root of unity is also an $mn$-th root of unity, hence of the form $\zeta_{mn}^r$ for some $r$.  By the assertion just before \cite[Lem.~4.1]{Garnier}, every regular element $w$ in $G(m, 1, n)$ has order $k$ dividing $mn$, hence all the eigenvalues of $w$ are $k$th roots of unity.  In particular, $w$ has as regular eigenvalue $\zeta_{mn}^r$ for some $r$.  And, as stated immediately following \cite[Def.~1.1]{Garnier}, therefore, $w$ is conjugate to $\epsilon^r$.  That is, every regular element in $G(m, 1, n)$ (hence in $W$) is conjugate to a power of $\epsilon$.  Moreover, it's easy to see that $\epsilon^r$ has of $\gcd(r, n)$ cycles, each of length $\frac{n}{\gcd(r, n)}$ and of weight $\frac{r}{\gcd(r, n)}$, and these data characterize its conjugacy class.

    In the group $W = G(m, m, n)$, all those elements of $W$ that are regular in $G(m, 1, n)$ are also regular in $W$ (since every reflecting hyperplane for $W$ is also a reflecting hyperplane for $G(m, 1, n)$).  However, we must also consider the possibility of elements of $W$ with regular eigenvectors that lie on one of the coordinate hyperplanes $x_i = 0$ (which are reflecting hyperplanes for $G(m, 1, n)$, but not for $W$).  Observe that a regular vector for $W$ can lie on at most one coordinate hyperplane (since the intersection of the planes $x_i = 0$ and $x_j = 0$ is contained in the plane $x_i - x_j = 0$, a reflecting hyperplane for $W$).  Therefore, if $w \in G(m, m, n)$ has a regular eigenvector $v$ that lies on the hyperplane $x_n = 0$, all other coordinates of $v$ (in the standard basis) must be nonzero, and so the underlying permutation of $w$ must have $n$ as a fixed point.  
    
    The subgroup of $G(m, m, n)$ consisting of elements whose underlying permutation has $n$ as a fixed point is isomorphic as a group to $G(m, 1, n - 1)$:  the map 
    \begin{equation}\label{eq:inclusion}
        [u; (a_1, \ldots, a_{n - 1})] \mapsto [u (n); (a_1, \ldots, a_{n - 1}, -(a_1 + \ldots + a_n)]
    \end{equation}
    is an isomorphism.  This map does not preserve reflections (specifically, the images in $W$ of diagonal reflections in $G(m, 1, n - 1)$ have two nontrivial diagonal entries), and indeed its codomain is not a reflection group on $\CC^n$.  However, when its action is restricted to the plane $x_n = 0$, the action of the subgroup agrees with the action of $G(m, 1, n - 1)$.  Furthermore, the regular vectors for $W$ in the hyperplane $x_n = 0$ are precisely the same as the regular vectors for this copy of $G(m, 1, n - 1)$: the reflections in $W$ but not in the subgroup, with reflecting planes $x_i - \xi x_n = 0$, have as their traces in the plane $x_n = 0$ the planes with relative equation $x_i = 0$, so they make up for the ``missing'' diagonal reflections.  Consequently, the regular elements of $G(m, m, n)$ with regular eigenvector in the hyperplane $x_n = 0$ are the images of the regular elements in $G(m, 1, n - 1)$ under the inclusion \eqref{eq:inclusion}; that is, they are the powers of $\epsilon' = [(1 \ 2 \ \cdots \ n - 1)(n); (0, \ldots, 0, 1, -1)]$ and the conjugates thereof.\footnote{The element $\epsilon'$ is a Coxeter element for $G(m, m, n)$, see \cite[\S2.2]{BessisCorran}.}  It is easy to see that $(\epsilon')^r$ consists of $\gcd(r, n - 1) + 1$ cycles, of which all but one have length $\frac{n - 1}{\gcd(r, n - 1)}$ and weight $\frac{r}{\gcd(r, n - 1)}$, and the last is a $1$-cycle with weight $-r$.  These data characterize its conjugacy class.
\end{proof}

\begin{proof}[Proof of Proposition~\ref{prop:regular}]
    First consider the case $p < m$, and let $w$ be a regular element in $G(m, p, n)$.  By Lemma~\ref{lem:regular elements}, $w$ has some number of cycles, all of the same weight.  By Corollary~\ref{cor:shi}, if $\lR(w) = \codimfix(w)$ then either the common weight of these cycles is $0 \pmod{p}$ (so there are no cycles whose weight is not divisible by $p$) or it is $\frac{m}{2}$ (so that the cycle weights can be paired to sum to $0$).  Thus, any subset of the cycle weights of $w$ that sums to $0 \pmod{p}$ either consists entirely of elements that are $0 \pmod{p}$ or consists of an even number of copies of $\frac{m}{2}$.  Then it immediately follows from Theorem~\ref{thm:main} that $[\id, w]_{\lR} = [\id, w]_{\cdf}$.

    Now consider the case $p = m$, and let $w$ be a regular element in $G(m, m, n)$.  If $w$ is in the first category in Lemma~\ref{lem:regular elements} (a power of a Coxeter element for $G(m, 1, n)$), the same argument as in the preceding paragraph works.  Otherwise, the multiset of cycle weights of $w$ contains, for some integers $r$, $g$ with $g \mid r$, $g$ copies of $r/g$ and one copy of $-r$.  We proceed in three cases.
    
    First, suppose $r/g \equiv 0 \pmod{m}$. Then all cycle weights of $w$ are $0$.  In this case it immediately follows from Theorem~\ref{thm:main} that $[\id, w]_{\lR} = [\id, w]_{\cdf}$.

    Second, suppose $g = 1$.  Then $w$ has only two cycles, whose weights sum to $0$, and it immediately follows from Theorem~\ref{thm:main} that $[\id, w]_{\lR} = [\id, w]_{\cdf}$.

    Third, suppose $g > 1$ and $r/g \not \equiv 0 \pmod{m}$.  Since $\lR(w) = \codimfix(w)$, we have by Corollary~\ref{cor:shi} that the multiset of cycle weights of $w$ that are not $0 \pmod{m}$ can be partitioned into pairs that sum to $0 \pmod{m}$.  By hypothesis, there are at least two cycles of weight $r/g$ (which is not $0 \pmod{m}$), and at most one of these can be paired with the cycle of weight $-r$; consequently, in this pairing, two cycles of weight $r/g$ must be paired together.  It follows that $r/g \equiv m/2 \pmod{m}$, and therefore that all cycle weights of $w$ are equal (in $\ZZ/m\ZZ$) to either $m/2$ or $0$.  Given a subset of such a multiset of sum $0$, the nonzero values may always be partitioned into pairs that sum to $0$; thus, by Theorem~\ref{thm:main}, we have that $[\id, w]_{\lR} = [\id, w]_{\cdf}$.

    Since the three cases above are exhaustive, the result is proved for $G(m, m, n)$.
\end{proof}

\begin{remark}\label{rem:exceptional}
    It is natural to ask whether the statement of Proposition~\ref{prop:regular} is furthermore true for all irreducible complex reflection groups.  Unfortunately, the answer is negative.  In particular, by an exhaustive check in SageMath \cite{sagemath}, using its interface with GAP and the CHEVIE package \cite{GAP, chevie}, we checked for each exceptional complex reflection group $W$ and each conjugacy class representative $w \in W$ whether $\lR(w) = \codimfix(w)$ but $[\id, w]_{\lR} \neq [\id, w]_{\cdf}$.  The result of this computation was that each of the groups $G_{31}$, $G_{32}$, and $G_{34}$ has two conjugacy classes of elements with these properties.  Of these six conjugacy classes, two of them are regular: in $G_{31}$, the (unique) class of elements with eigenvalues $\left\{\exp\left(2\pi i \cdot \frac{1}{6}\right), \exp\left(2\pi i \cdot \frac{1}{6}\right), \exp\left(2\pi i \cdot \frac{5}{6}\right), \exp\left(2\pi i \cdot \frac{5}{6}\right)\right\}$ is regular and satisfies $\lR = \codimfix = 4$, but the interval $[\id, w]_{\lR}$ contains $128$ elements while the interval $[\id, w]_{\cdf}$ contains $134$ elements; and in $G_{32}$, the (unique) class of elements with eigenvalues $\{i, i, -i, -i\}$ is regular and satisfies $\lR = \codimfix = 4$, but the interval $[\id, w]_{\lR}$ contains $104$ elements while the interval $[\id, w]_{\cdf}$ contains $108$ elements.
\end{remark}

\subsubsection{Parabolic quasi-Coxeter elements}

An element of a reflection group $W$ is \emph{quasi-Coxeter} if it has a minimum-length reflection factorization whose factors generate $W$.  An element is \emph{parabolic quasi-Coxeter} if it is a quasi-Coxeter element for a parabolic subgroup of $W$.  These elements have an assortment of interesting properties and characterizations \cite{DLM2}; among these is that if $w$ is a parabolic quasi-Coxeter element, then $\lR(w) = \codimfix(w)$ \cite[Prop.~3.3]{DLM2}.  Our final result, a considerable strengthening of \cite[Cor.~6.6]{LM2021}, shows that these properties include the equality of reflection length and fixed space codimension intervals.

\begin{thm}
    If $W$ is a well generated complex reflection group and $w$ is a parabolic quasi-Coxeter element in $W$, then $[\id, w]_{\lR} = [\id, w]_{\cdf}$.
\end{thm}

\begin{proof}
Since reducible well generated groups are products of irreducible well generated groups, and an element is parabolic quasi-Coxeter for a product if and only if each of its components is parabolic quasi-Coxeter in the corresponding factor, the result holds if it holds for irreducible groups.  We proceed case-by-case.

By \cite[Cor.~3.16]{DLM2}, the parabolic quasi-Coxeter elements of the groups in the infinite family may be characterized combinatorially as follows: in $G(m, 1, n)$, they are the elements that have at most one cycle of nonzero weight, and if there is such a cycle, its weight must be primitive modulo $m$; in $G(m, m, n)$, they are the elements that have at most two cycles of nonzero weight, and if there are cycles of nonzero weight, their weights must be primitive modulo $m$ (and must sum to $0$, by the definition of $G(m, m, n)$).  It is an immediate consequence of Theorem~\ref{thm:main} that all such elements satisfy $[\id, w]_{\lR} = [\id, w]_{\cdf}$.

As observed in Remark~\ref{rem:exceptional}, there are only six conjugacy conjugacy classes of elements $w$ in the exceptional groups that satisfy $\lR(w) = \codimfix(w)$ but $[\id, w]_{\lR} \neq [\id, w]_{\cdf}$.  Of these, two belong to the badly generated group $G_{31}$, so we discard them.  For each of the remaining four pairs $(w, W)$, we used Sage/GAP/Chevie to first verify that $\lR(w) = \rank(W)$, then to produce a minimum-length reflection factorization of $w$, and finally to confirm that the factors in this factorization generate a proper subgroup of $W$.  By \cite[Thm.~5.6]{LW}, it follows that \emph{every} minimum-length reflection factorization of $w$ generates a proper subgroup, and hence that $w$ is not parabolic quasi-Coxeter for $W$.  Taking the contrapositive, every parabolic quasi-Coxeter element $w$ in an exceptional group $W$ satisfies $[\id, w]_{\lR} = [\id, w]_{\cdf}$.
\end{proof}

In the case of \emph{real} reflection groups, the parabolic quasi-Coxeter elements are precisely the elements whose minimum-length reflection factorizations form a single orbit under the \emph{Hurwitz action} (a natural action of the braid group on factorizations) \cite{BGRW}.  This equivalence does not hold in the complex case (see \cite[Ex.~4.3]{LW}), but one can ask whether the Hurwitz-transitive elements share the nice properties of parabolic quasi-Coxeter elements in this context.  The answer to this question is negative; for example, $w := [\id;(1,2,3,-1,-2,-3)]\in G(7,7,6)$ has $\lR(w) = \codimfix(w)$ and is Hurwitz-transitive, but $[\id, w]_{\cdf} \neq [\id, w]_{\lR}$.

\subsection{The main theorem in the cases that $\lR = \codimfix$ on $W$}

In the case that $W$ satisfies $\lR(w) = \codimfix(w)$ for all $w \in W$, of course the posets defined by the two functions are the same, and hence the lower intervals they determine are also the same.  We quickly record how to read this fact from the statement of our main Theorem~\ref{thm:main}.  The group $G(m, p, n)$ satisfies $\lR(w) = \codimfix(w)$ for all $w$ in the follows cases:

In the case that $p = 1$, the first condition in~\Cref{thm:main} is vacuous (every cycle weight is $0 \pmod{p}$) and the second condition is trivial (every subset is a disjoint union of singletons).

In the case that $m = p = 2$, both conditions amount to the observation that when a sum of integers is even, there must be an even number of odd summands, and hence the odd summands can be paired off.

In the case that $m = p$ and $n = 2$, the element $w$ can have either one cycle of weight $0$ or two cycles whose weights sum to $0$, so the partition in question is always the one that includes all cycles in a single part.

\subsection{Subadditive functions on the symmetric group}\label{sec:experiment}

The general results of Section~\ref{sec:general} raise the question of whether there are other natural subadditive functions on interesting groups.  The FindStat database \cite{FindStat} includes a large number of interesting statistics on the symmetric group and hyperoctahedral group.  Martin Rubey reports (personal communication) the following results of a comprehensive examination of these statistics.

Of the $400$ permutation statistics in FindStat (as of January 15, 2024), precisely $22$ of them are subadditive when restricted to $\Symm_n$ for $n \leq 6$ and take the value $0$ only on the identity.  Of these $22$, most are ``trivially'' subadditive: nine (namely, 
\FindStat{St000018}
, \FindStat{St000216}
, \FindStat{St000670}
, \FindStat{St001076}
, \FindStat{St001077}
, \FindStat{St001079}
, \FindStat{St001080}
, \FindStat{St001375}
, and \FindStat{St001760}
) are explicitly defined as a length function on $\Symm_n$ with respect to some generating set; two (\FindStat{St000673} 
 and \FindStat{St000829}
) are defined as a distance from the origin in an appropriate metric (as in Remark~\ref{rem:metric}; respectively, the Hamming distance and the Ulam distance, see \cite[\S2]{ChatterjeeDiaconis}); one (\FindStat{St000029}, the \emph{depth} of a permutation---see \cite{PT}) is a weighted length function, i.e., it is the minimum cost of a factorization in terms of some generating set, where the generators may have different costs; and another (\FindStat{St000830}
) has the property that $\text{St000830}(w) = 2 \cdot \text{St000029}(w)$ for every permutation $w$, hence inherits subadditivity from St000029.  This leaves nine statistics that seem to be subadditive for what might be called non-trivial reasons:
\begin{itemize}
\item \FindStat{St000019}, the cardinality of the support of a permutation (in the sense of reduced words),
\item \FindStat{St000141}, the maximum drop size of a permutation,
\item \FindStat{St000155}, the number of excedances of a permutation,
\item \FindStat{St000316}, the number of non-left-to-right-maxima of a permutation,
\item \FindStat{St000653}, the last descent of a permutation,
\item \FindStat{St000703}, the number of deficiencies of a permutation,
\item \FindStat{St000956}, the maximal displacement of a permutation,
\item \FindStat{St001569}, the maximal modular displacement of a permutation, and
\item \FindStat{St001759}, the Rajchgot index of a permutation.
\end{itemize}
The case of the excedance number 
\[
\exc(w) := \# \{i \in [n - 1] \colon w(i) > i\}
\]
is particularly striking.
\begin{prop}\label{prop:exc}
    A permutation $w$ has $\exc(w) = k$ if and only if $w$ can be written as a product of $k$ $1$-excedance permutations, and no fewer; that is, $\exc$ is the length function for $\Symm_n$ generated by $1$-excedance permutations.  Furthermore, the poset $(\Symm_n, \leq_{\exc})$ defined (as in Proposition~\ref{prop:poset}) by $x \leq_{\exc} y$ if and only if $\exc(x) + \exc(x^{-1}y) = \exc(y)$ is graded, with rank sizes given by the Eulerian numbers, and has unique maximal element $2 3 \cdots n 1 = (1 \ 2 \ \cdots \ n)$.
\end{prop}
\begin{remark}
    It is easy to see that the $1$-excedance permutations are precisely those whose cycle notation consists of a single nontrivial cycle $(a_1 \ a_2 \ \cdots \ a_k)$ in which the cyclic order can be chosen such that $a_1 > a_2 > \ldots > a_k$.
\end{remark}
\begin{proof}[Proof of Proposition~\ref{prop:exc}]
First, suppose $x$ and $y$ are permutations and that $x y$ has excedences $i_1$, \ldots, $i_k$.  Thus $x(y(i_1)) > i_1$, \ldots, $x(y(i_k)) > i_k$.  Observe that for $j = 1, \ldots, k$, we cannot have both $y(i_j) \leq i_j$ and $x(y(i_j)) \leq y(i_j)$; thus, for each $j$, either $i_j$ is an excedence of $y$ or $y(i_j)$ is an excedance of $x$.  Thus, $\exc(x) + \exc(y) \geq k = \exc(xy)$, so the excedance number is subadditive.  It follows that a permutation with $k$ excedances cannot be written as a product of fewer than $k$ $1$-excedance permutations.  To show that $\exc$ is a length function, it remains to show that each permutation with $k$ excedances is a product of $k$ $1$-excedance permutations.  We now explicitly construct such a factorization.

Consider the following Elementary Fact: for any permutation $w$ and any position $a$, if $w(a) > a$ then there must be some position $b > a$ with $w(b) \leq a$.  Step 1: Suppose that $w$ has an excedance in position $a_1$.  By the Elementary Fact, there is some position $a_2 > a_1$ with $w(a_2) \leq a_1$.  Let $w' = w (a_1 \ a_2)$ (and observe that $w(a_1) = w'(a_2)$).  If $w'(a_2) \leq a_2$ then go on to Step 2.  Otherwise, by the Elementary Fact, there is some position $a_3 > a_2$ with $w'(a_3) \leq a_2$.  Let $w'' = w' (a_2 \ a_3)$ (and observe that $w(a_1) = w''(a_3)$).  If $w''(a_3) \leq a_3$ then go to Step 2.  Otherwise continue in the same way.  Step 2: We have eventually produced a new permutation $v = w (a_1 \ a_2 \ \cdots \ a_k)$ with $a_1 < \ldots < a_k$ and $v(a_k) \leq a_k$.  By construction, $v$ has one fewer excedance than $w$, and so $w = v \cdot (a_k \ \cdots \ a_1)$ is an expression for $w$ as a product of a $(\exc(w) - 1)$-excedance permutation with a $1$-excedance permutation.  By induction, $w$ can be written as a product of $\exc(w)$ $1$-excedance permutations, as claimed.

It is easy to see that the permutation $2 3 \cdots n 1 = (1 \ 2 \ \cdots \ n)$ is the unique element of $\Symm_n$ with $n - 1$ excedances, and it is straightforward to adjust the argument of the previous paragraph to show that no other element of $\Symm_n$ is maximal in the $\leq_{\exc}$-order.  The fact that the poset $(\Symm_n, \leq_{\exc})$ is graded with rank sizes given by Eulerian numbers follows immediately from Proposition~\ref{prop:graded} and the well known distribution of excedances \cite[\S1.3]{Petersen}.
\end{proof}

Of course the same analysis applies to deficiencies.  Of the remaining statistics, the Rajchgot index $\op{raj}(w)$ seems particularly intriguing.  This permutation statistic may be defined as the degree of the Grothendieck polynomial $\mathfrak{G}_w$, or by the formula  
\[
\op{raj}(w) := \max \{ \op{maj}(v) \colon v \preceq w \},
\]
where $\op{maj}$ denotes the major index and $\preceq$ denotes the right weak order on $\Symm_n$ \cite{PSW}.
We see no obvious reason that $\op{raj}$ should be subadditive.

For the group $\Symm^{\pm}_n = G(2, 1, n)$ of signed permutations, viewed as bijections from $\{\pm 1, \ldots, \pm n\}$ to itself that satisfy $w(-i) = -w(i)$ for all $i$, FindStat contains $45$ statistics.  Of these, Rubey's calculations show that four are subadditive when restricted to $\Symm^{\pm}_n$ for $n \leq 4$ and have the property that $f(w) = 0 \Leftrightarrow w = \id$.  These properties are immediate from the definition for three of them: \FindStat{St001428} (the Coxeter length, from which one recovers the weak order), \FindStat{St001769} (the reflection length, from which one recovers the absolute order), and \FindStat{St001894} (the depth).  The remaining example is \FindStat{St001907}, an analog $\exc_B$ of the excedance statistic\footnote{Several other excedance-like statistics have been proposed for the group of signed permutations, see \cite[Rem.~6.2]{BHS}.} defined by Bastidas--Hohlweg--Saliola \cite{BHS}.  It is defined by
\[
\exc_B(w) = \left\lfloor \frac{2 \cdot \# \{i \in [n - 1] \colon w(i) > i\} + \#\{i \in [n] \colon w(i) < 0\} + 1}{2}\right\rfloor.
\]
Our calculations for $n \leq 4$ suggest that, as for the usual excedance statistic $\exc$ on $\Symm_n$, $\exc_B$ is in fact a length function on $\Symm^{\pm}_n$ with respect to the generating set $\{ w \in \Symm^{\pm}_n \colon \exc_B(w) = 1\}$ and so (by Proposition~\ref{prop:graded}) that the associated poset is graded by $\exc_B$, with rank sizes given by the type-B Eulerian numbers and a unique maximal element $2 \ 3\ \cdots\  n\  \ol{1}$.  We have not attempted to prove these assertions.

We also remark that Rubey's calculations show that, with the exception of the Coxeter length, none of these statistics produce a lattice; the Coxeter length recovers the weak order.

\section*{Acknowledgements}
We thank Theo Douvropoulos, William Dugan, Alejandro Morales, Ga Yee Park, and Jesse Selover for interesting discussions, and Ivan Marin for his assistance in navigating the literature.  We are grateful to Martin Rubey for conceiving of the experiments described in Section~\ref{sec:experiment}, sharing the results of these experiments with us, and permitting us to describe them here.

\bibliography{my}

\begin{thebibliography}{10}

\bibitem{BHS}
Jose Bastidas, Christophe Hohlweg, and Franco Saliola.
\newblock The primitive {E}ulerian polynomial.
\newblock \href{https://arxiv.org/abs/2306.15556}{\texttt{arXiv:2306.15556v1}}, 2023.

\bibitem{BGRW}
Barbara Baumeister, Thomas Gobet, Kieran Roberts, and Patrick Wegener.
\newblock On the {H}urwitz action in finite {C}oxeter groups.
\newblock {\em J. Group Theory}, 20(1):103--131, 2017.
\newblock \href {https://doi.org/10.1515/jgth-2016-0025} {\path{doi:10.1515/jgth-2016-0025}}.

\bibitem{BessisCorran}
David Bessis and Ruth Corran.
\newblock Non-crossing partitions of type {$(e,e,r)$}.
\newblock {\em Adv. Math.}, 202(1):1--49, 2006.
\newblock \href {https://doi.org/10.1016/j.aim.2005.03.004} {\path{doi:10.1016/j.aim.2005.03.004}}.

\bibitem{BMS}
Mireille Bousquet-M\'{e}lou and Gilles Schaeffer.
\newblock Enumeration of planar constellations.
\newblock {\em Adv. in Appl. Math.}, 24(4):337--368, 2000.
\newblock \href {https://doi.org/10.1006/aama.1999.0673} {\path{doi:10.1006/aama.1999.0673}}.

\bibitem{BradyWatt}
T.~Brady and C.~Watt.
\newblock A partial order on the orthogonal group.
\newblock {\em Comm. Algebra}, 30(8):3749--3754, 2002.
\newblock \href {https://doi.org/10.1081/AGB-120005817} {\path{doi:10.1081/AGB-120005817}}.

\bibitem{broue_book}
M.~Brou\'{e}.
\newblock {\em Introduction to complex reflection groups and their braid groups}, volume 1988 of {\em Lecture Notes in Mathematics}.
\newblock Springer-Verlag, Berlin, 2010.
\newblock \href {https://doi.org/10.1007/978-3-642-11175-4} {\path{doi:10.1007/978-3-642-11175-4}}.

\bibitem{Carter}
R.W. Carter.
\newblock Conjugacy classes in the {W}eyl group.
\newblock {\em Compositio Math.}, 25:1--59, 1972.

\bibitem{ChatterjeeDiaconis}
Sourav Chatterjee and Persi Diaconis.
\newblock A central limit theorem for a new statistic on permutations.
\newblock {\em Indian J. Pure Appl. Math.}, 48(4):561--573, 2017.
\newblock \href {https://doi.org/10.1007/s13226-017-0246-3} {\path{doi:10.1007/s13226-017-0246-3}}.

\bibitem{Dieudonne}
J.~Dieudonn\'e.
\newblock Sur les g\'en\'erateurs des groupes classiques.
\newblock {\em Summa Brasil. Math.}, 3:149--179, 1955.

\bibitem{DLM2}
T.~Douvropoulos, J.B. Lewis, and A.H. Morales.
\newblock {H}urwitz numbers for reflection groups {II}: parabolic quasi-{C}oxeter elements.
\newblock {\em J. Algebra}, 641:648--715, 2024.
\newblock \href {https://doi.org/10.1016/j.jalgebra.2023.11.015} {\path{doi:10.1016/j.jalgebra.2023.11.015}}.

\bibitem{FosterGreenwood}
B.~Foster-Greenwood.
\newblock Comparing codimension and absolute length in complex reflection groups.
\newblock {\em Comm. Algebra}, 42(10):4350--4365, 2014.
\newblock \href {https://doi.org/10.1080/00927872.2013.810748} {\path{doi:10.1080/00927872.2013.810748}}.

\bibitem{GAP}
The GAP~Group.
\newblock {\em {GAP -- Groups, Algorithms, and Programming, Version 4.12.2}}, 2022.
\newblock URL: \url{http://www.gap-system.org}.

\bibitem{Garnier}
Owen Garnier.
\newblock Regular theory in complex braid groups.
\newblock {\em J. Algebra}, 620:534--557, 2023.
\newblock \href {https://doi.org/10.1016/j.jalgebra.2022.12.031} {\path{doi:10.1016/j.jalgebra.2022.12.031}}.

\bibitem{chevie}
M.~Geck, G.~Hiss, F.~L{\"u}beck, G.~Malle, and G.~Pfeiffer.
\newblock {\sf CHEVIE} -- {A} system for computing and processing generic character tables for finite groups of {L}ie type, {W}eyl groups and {H}ecke algebras.
\newblock {\em Appl. Algebra Engrg. Comm. Comput.}, 7:175--210, 1996.
\newblock \href {https://doi.org/10.1007/BF01190329} {\path{doi:10.1007/BF01190329}}.

\bibitem{HLR}
J.~Huang, J.B. Lewis, and V.~Reiner.
\newblock Absolute order in general linear groups.
\newblock {\em J. Lond. Math. Soc. (2)}, 95(1):223--247, 2017.
\newblock \href {https://doi.org/10.1112/jlms.12013} {\path{doi:10.1112/jlms.12013}}.

\bibitem{LM2021}
Joel~Brewster Lewis and Alejandro~H. Morales.
\newblock Factorization problems in complex reflection groups.
\newblock {\em Canadian J. Math.}, 73(4):899–946, 2021.
\newblock \href {https://doi.org/10.4153/S0008414X2000022X} {\path{doi:10.4153/S0008414X2000022X}}.

\bibitem{LW}
Joel~Brewster Lewis and Jiayuan Wang.
\newblock The {H}urwitz action in complex reflection groups.
\newblock {\em Comb. Theory}, 2(1), 2022.
\newblock \href {https://doi.org/10.5070/c62156884} {\path{doi:10.5070/c62156884}}.

\bibitem{PSW}
O.~Pechenik, D.~E. Speyer, and A.~Weigandt.
\newblock {C}astelnuovo--{M}umford regularity of matrix {S}chubert varieties.
\newblock \href{https://arxiv.org/abs/2111.10681}{\texttt{arXiv:2111.10681v1}}, 2021.

\bibitem{Petersen}
T.~Kyle Petersen.
\newblock {\em Eulerian numbers}.
\newblock Birkh\"{a}user Advanced Texts: Basler Lehrb\"{u}cher. [Birkh\"{a}user Advanced Texts: Basel Textbooks]. Birkh\"{a}user/Springer, New York, 2015.
\newblock With a foreword by Richard Stanley.
\newblock \href {https://doi.org/10.1007/978-1-4939-3091-3} {\path{doi:10.1007/978-1-4939-3091-3}}.

\bibitem{PT}
T.~Kyle Petersen and Bridget~Eileen Tenner.
\newblock The depth of a permutation.
\newblock {\em J. Comb.}, 6(1-2):145--178, 2015.
\newblock \href {https://doi.org/10.4310/JOC.2015.v6.n1.a9} {\path{doi:10.4310/JOC.2015.v6.n1.a9}}.

\bibitem{RRS}
V.~Reiner, V.~Ripoll, and C.~Stump.
\newblock On non-conjugate {C}oxeter elements in well-generated reflection groups.
\newblock {\em Math. Z.}, 285(3-4):1041--1062, 2017.
\newblock \href {https://doi.org/10.1007/s00209-016-1736-4} {\path{doi:10.1007/s00209-016-1736-4}}.

\bibitem{FindStat}
Martin Rubey, Christian Stump, et~al.
\newblock {FindStat} - {T}he combinatorial statistics database.
\newblock Accessed January 15, 2024.
\newblock URL: \url{http://www.FindStat.org}.

\bibitem{ShephardTodd}
Geoffrey~C Shephard and John~A Todd.
\newblock Finite unitary reflection groups.
\newblock {\em Canadian J. Math.}, 6:274--304, 1954.
\newblock \href {https://doi.org/10.4153/cjm-1954-028-3} {\path{doi:10.4153/cjm-1954-028-3}}.

\bibitem{Shi2007}
Jian-yi Shi.
\newblock Formula for the reflection length of elements in the group {$G(m,p,n)$}.
\newblock {\em J. Algebra}, 316(1):284--296, 2007.
\newblock \href {https://doi.org/10.1016/j.jalgebra.2007.06.031} {\path{doi:10.1016/j.jalgebra.2007.06.031}}.

\bibitem{sagemath}
{The Sage Developers}.
\newblock {\em {S}ageMath, the {S}age {M}athematics {S}oftware {S}ystem ({V}ersion 10.0)}, 2023.
\newblock URL: \url{https://www.sagemath.org}.

\end{thebibliography}

\end{document}